\documentclass[11pt, reqno]{amsart}

\usepackage{graphicx}
\usepackage{amsmath,amssymb,mathtools,amsopn,amsfonts,amsthm}
\usepackage[shortlabels]{enumitem}
\usepackage{float}
\usepackage{bbm}
\usepackage{subfigure}
\usepackage{eepic}
\usepackage{tikz}
\usetikzlibrary{positioning}
\usepackage{a4wide}
\usepackage[utf8]{inputenc}
\usepackage{epsfig}
\usepackage{latexsym}
\usepackage[UKenglish]{babel}
\usepackage{verbatim}
\usepackage[initials]{amsrefs}
\usepackage[bookmarks]{hyperref}
\usetikzlibrary{decorations.markings}

\newtheorem{thm}{Theorem}[section]
\newtheorem{lma}[thm]{Lemma}

\newtheorem{prop}[thm]{Proposition}
\theoremstyle{definition}

\newtheorem{rem}[thm]{Remark}
\newtheorem{ques}[thm]{Question}

\newtheorem*{theorem*}{Theorem}

\newcommand{\R}{\mathbb{R}}

\newcommand{\N}{\mathbb{N}}

\providecommand{\norm}[1]{\lVert#1\rVert}

\newcommand{\I}{\mathcal{I}}
\newcommand{\J}{\mathcal{J}}
\renewcommand{\i}{\mathtt{i}}
\renewcommand{\j}{\mathtt{j}}
\renewcommand{\k}{\mathtt{k}}
\renewcommand{\l}{\mathtt{l}  }

\newcommand{\pd}{\dim_\textup{P}}
\newcommand{\hd}{\dim_\textup{H}}
\newcommand{\bd}{\dim_\textup{B}}

\renewcommand{\o}{\overline{\alpha}}
\renewcommand{\u}{\underline{\alpha}}
\renewcommand{\epsilon}{\varepsilon}

\numberwithin{equation}{section}

\title{Parabolic carpets}

\author{Jonathan M. Fraser} \address{Mathematical Institute, University of St Andrews, Scotland, KY16 9SS}
\email{jmf32@st-andrews.ac.uk}
\author{Natalia Jurga} \address{Mathematical Institute, University of St Andrews, Scotland, KY16 9SS}
\email{naj1@st-andrews.ac.uk}

\thanks{The  authors were both     supported by an \emph{EPSRC Standard Grant} (EP/R015104/1). J. M. Fraser was also supported by  a  \emph{Leverhulme Trust Research Project Grant} (RPG-2019-034) and an \emph{RSE Sabbatical Research Grant} (70249).}

\begin{document}

\maketitle

\begin{abstract}
We introduce and study a family of  non-conformal and non-uniformly contracting iterated function systems.  We refer to the attractors of such systems as \emph{parabolic carpets}.  Roughly speaking they may be thought of as   nonlinear analogues of  self-affine carpets which are allowed to have parabolic fixed points.  We compute the $L^q$-spectrum of a class of weak Gibbs measures supported on parabolic carpets as well as the box dimensions of the carpet itself.  
\end{abstract} 

 \tableofcontents

\section{Introduction}

\subsection{Parabolic carpets and parabolic IFSs}

Fractals generated by hyperbolic conformal dynamical systems are among the most well-understood and well-studied objects in fractal geometry.  Examples include self-similar sets and self-conformal sets.  There are two key generalisations where the resulting dynamical systems and associated fractals become much more complicated and give rise to many new features.  The first approach  is to drop the conformality assumption.  The simplest examples of such systems are self-affine sets, where the defining maps can distort space by different amounts in different directions. Self-affine carpets, e.g. those studied by Bedford-McMullen \cite{bedford, mcmullen}, are the simplest example of self-affine sets, but nevertheless give rise to a host of interesting properties.  The second approach is to drop the hyperbolicity assumption.  This gives rise to \emph{parabolicity}, which appears in many contexts, including: parabolic Cantor sets, parabolic Julia sets and limit sets of Kleinian groups.  Perhaps the simplest examples are invariant sets for parabolic interval maps  such as the Manneville-Pomeau system but, again, many technical difficulties and new phenomena arise in this setting.  Our idea in this paper is to blend these approaches, studying for the first time, parabolic non-conformal systems.

We say that a map $h$ on $[0,1]$ is a \emph{contraction} if, for all distinct $x,y \in [0,1]$, $|h(x)-h(y) | < |x-y|$.  Throughout we assume $h$ is differentiable and   use one-sided derivatives at the end points without explicitly drawing attention to it.  The important feature of our definition of contraction is that parabolic points are allowed.  We say $p \in [0,1]$ is a \emph{parabolic point} for differentiable  $h$ if $h(p)=p$ and $|h'(p)| = 1$.  Clearly if $h$ is a differentiable  contraction then it has at most one parabolic point.

Consider a set of (at least two) contractions $\{S_i\}_{i \in \I}$ acting on $X$ (which will either be $[0,1]^2$ or $[0,1]$). We refer to such a collection as a parabolic IFS.   Let $\Sigma = \mathcal{I}^\infty$ be the space of all infinite words over $\mathcal{I}$ and, for $\i = i_1  i_2 \dots \in \Sigma$ and $n \in \mathbb{N}$, define $\i|n = i_1 \dots i_n$   and
\[
S_{\i|n} = S_{i_1} \circ \cdots \circ S_{i_n}.
\]
Then define a map $\Pi: \Sigma \to X$ by 
\[
\{\Pi(\i)\} = \bigcap_{n=1}^\infty S_{\i|n}([0,1]^2).
\]
The fact that the maps are contractions ensures that $\Pi$ is well-defined since, for all $\i \in \Sigma$, $\Pi(\i) \in X$ is a  single point.  The \emph{attractor} of a parabolic IFS is then defined by
\[
F = \Pi(\Sigma).
\]
We will be interested in measures supported on attractors of parabolic IFSs.  Consider a Borel probability measure $\mathbb{P}$ supported on $\Sigma$ where $\Sigma$ is equipped with the product topology.  Let $\mathcal{I}^*$ denote the set of all finite words with digits in $\I$. For $\i \in \mathcal{I}^*$ we write $[\i]$ for the cylinder consisting of all infinite words starting with $\i$.  Our basic assumption will be that $\mathbb{P}$   satisfies the \emph{weak quasi-Bernoulli property}: there exists a sequence $(c_n)_{n \in \N}$ of positive real numbers satisfying $\lim_{n \to \infty} c_n^{\frac{1}{n}}=1$ such that for all $\i_{1} \in \I^{n_1}, \ldots, \i_{k} \in \I^{n_k}$,
\begin{equation} \label{wqbp}
c_{n_1}^{-1} \cdots c_{n_k}^{-1}\leq \frac{\mathbb{P}([\i_1 \ldots \i_k])}{\mathbb{P}([\i_1])\cdots \mathbb{P}([ \i_k])} \leq c_{n_1} \cdots c_{n_k}.
\end{equation}
This includes the class of weak Gibbs measures, see for instance \cite{jordanrams}. If $c_n$ can be taken to be a uniform constant over all $n \in \N$, then we say that $\mathbb{P}$ has the \emph{quasi-Bernoulli property}. Apart from assuming that $\mathbb{P}$ satisfies \eqref{wqbp} we will also need to assume that when $\mathbb{P}$ is restricted to a certain `induced subsystem', $\mathbb{P}$ has the quasi-Bernoulli property. We delay the discussion of this until  \S \ref{aqbp}.
Note that our measures are not necessarily invariant under the left-shift.  Now, 
\[
\mu = \mathbb{P} \circ \Pi^{-1}
\]
is a Borel probability measure supported on $F$.  We refer to such measures as \emph{weakly quasi-Bernoulli measures}.

Consider an IFS $\{S_i\}_{i \in \I}$ acting on $[0,1]^2$ where the maps are given by $S_i(x,y):=(f_i(x),g_i(y))$ and:
\begin{enumerate}
\item[A1]  there exists $\alpha_f, \alpha_g>0$ such that for all $i \in \I$,  $f_i$ are $C^{1+\alpha_f}$ contractions on $[0,1]$ and $g_i$ are $C^{1+\alpha_g}$ contractions on $[0,1]$.  \label{ifs}
\item[A2]  for all $i \in \I$,  $f_i, g_i$ have non-vanishing derivative on $[0,1]$.  \label{bilip}
\item[A3] for all $i,j \in \I$, if $f_i((0,1)) \cap f_j((0,1)) \neq \emptyset$, then $f_i \equiv f_j$  and if $g_i((0,1)) \cap g_j((0,1)) \neq \emptyset$, then $g_i \equiv g_j$. In particular if $f_i((0,1)) \cap f_j((0,1)) \neq \emptyset$ and $g_i((0,1)) \cap g_j((0,1)) \neq \emptyset$ then $i=j$.  (This forces the open set condition and also imposes a grid like structure on the IFS, similar to Bedford-McMullen carpets for example.)  \label{grid}
\end{enumerate}
We call  attractors of parabolic IFSs satisfying A1, A2 and A3  \emph{parabolic carpets}.

\begin{figure}[h]
							\includegraphics[width=\textwidth]{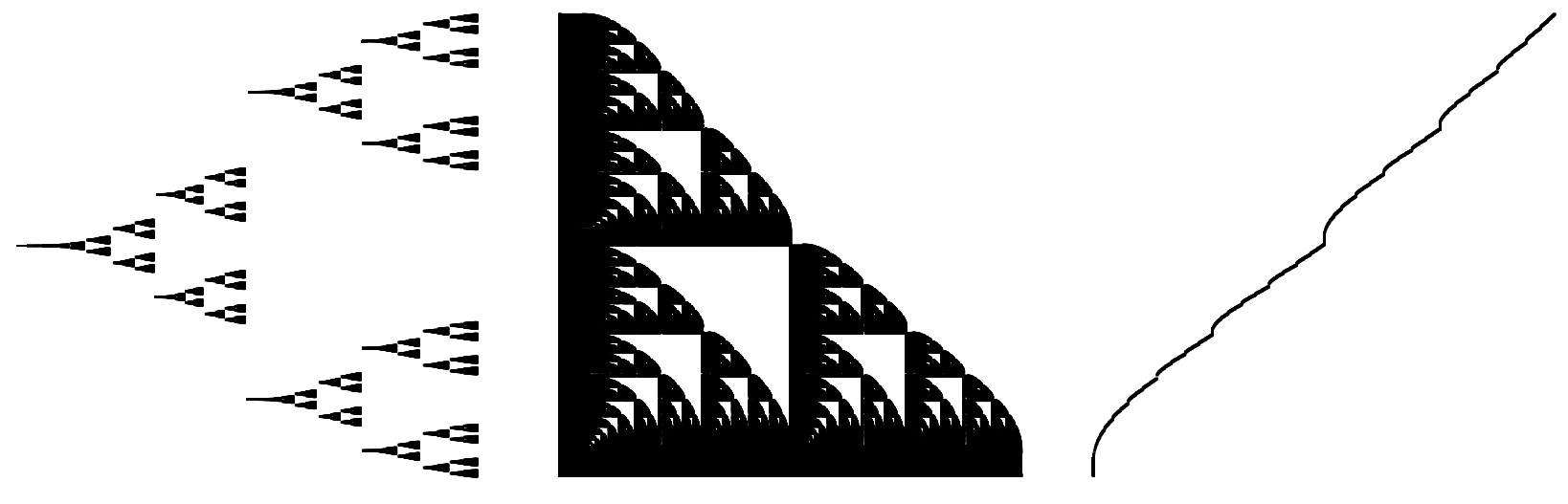} \\

			\caption{Three parabolic carpets.  Left: the IFS is constructed from the product of a Manneville-Pomeau system $x \mapsto f_\alpha(x)$ with parameter $\alpha=0.9$ (see \eqref{MP}) and the tripling map $y \mapsto 3y \textup{ mod 1}$. Centre: a parabolic variant on the right-angled Sierpi\'nski triangle.  The maps are all $C^2$ and the bottom left map is $x \mapsto   (\sqrt{1 + 8x}-1)/4$ in both coordinates. In particular, $(0,0)$ is a double parabolic fixed point.  Right: this set is invariant under the product of two Manneville-Pomeau systems with different parameters.}
			\end{figure}\label{examplesfig}

The projections of parabolic carpets onto their coordinate axes will be important in our analysis.  This is a common feature of self-affine carpets and is a consequence of cylinder sets becoming increasingly long and thin.  For this reason we are led to consider parabolic Cantor sets.  The definition is similar but we keep them separate for clarity. Consider an IFS $\{h_i\}_{i \in \I}$ acting on $[0,1]$ where:
\begin{enumerate}
\item[A1']  there exists $\alpha_h>0$ such that for all $i \in \I$,  $h_i$ are $C^{1+\alpha_h}$ contractions on $[0,1]$.  \label{ifs2}
\item[A2']  for all $i \in \I$,  $h_i$ has  non-vanishing derivative on  $[0,1]$.  \label{bilip2}
\item[A3'] for all $i,j \in \I$, if $h_i((0,1)) \cap h_j((0,1)) \neq \emptyset$, then $i=j$. \label{grid2}
\end{enumerate}
We call  attractors of parabolic IFSs satisfying  A1', A2' and A3' \emph{parabolic Cantor sets}.  In particular, if the attractor is not the whole interval then it is a topological Cantor set. Parabolic Cantor sets have been studied before in various contexts, sometimes with slightly different assumptions, see  for example \cite{gelfertrams, jordanrams, parifs, urbanski}.

\subsection{Dimensions of sets and measures}

We write $\hd$ and $\bd$ to denote the Hausdorff and box dimension respectively.  For a general compactly supported Borel probability measure  $\mu$   on  $  \mathbb{R}^d$, the $L^q$-spectrum of $\mu$ is a  function parametrised by $q \geq 0$ which measures coarse fluctuations in the measure.  This is captured by studying the growth rate of the moments
\[
D_\delta^q(\mu) = \sum_{Q \in \mathcal{Q}_\delta} \mu(Q)^q
\]
where $\mathcal{Q}_\delta$ is the collection of closed cubes in a $\delta$-mesh imposed on $\mathbb{R}^d$ oriented with the coordinate axes and we adopt the convention that $0^0=0$.  Note that for $q=0$ this is the number of   $\delta$-cubes intersecting the support of the measure.

Formally, for  $q\geq 0$ the upper and lower $L^q$-spectrum of $\mu$ are defined to be 
\begin{equation*}
\overline{\tau}_{\mu}(q)= \overline{\lim}_{\delta\rightarrow 0}\frac{\log D_\delta^q(\mu) }{-\log\delta}
\end{equation*}
and 
\begin{equation*}
\underline{\tau}_{\mu}(q)= \underline{\lim}_{\delta\rightarrow 0}\frac{\log D_\delta^q(\mu) }{-\log\delta}
\end{equation*}
respectively. If these two values coincide we define the $L^q$-spectrum  of $\mu$, denoted by $\tau_{\mu}(q)$, to be the common value.  The $L^q$-spectrum encodes lots of information about the measure and its support.  For example, $\overline{\tau}_{\mu}(0)$ and $\underline{\tau}_{\mu}(0)$ coincide with the upper and lower box dimensions of the support of $\mu$.  When it exists $-\tau_\mu'(1)$ gives the Hausdorff dimension of $\mu$ and the Legendre transform of $\tau_{\mu}$ provides an upper bound for the (increasing part of the) multifractal spectrum of $\mu$.  If  the multifractal formalism holds, then the Legendre transform of $\tau_{\mu}$ and the increasing part of the  multifractal spectrum of $\mu$ coincide.

\subsection{Notation}

We  write $A \lesssim B$ to mean there exists a constant $c>0$ such that $A \leq cB$.  We define $\gtrsim$ in an analogous way and $A \approx B$ will mean $A \lesssim B$ and $A \gtrsim B$.   If we want to emphasise that the constant $c$ depends on a parameter $w$, then we will write, for example,  $A \lesssim_w B$.  Sometimes it will be important to keep track of constants and so we will not exclusively use this notation.  When we do use it, it will be to make arguments cleaner by suppressing constants which do not play an important role. 

We write $|X|$ to denote the diameter of a non-empty set $X \subseteq \mathbb{R}^d$. This is not to be confused with the absolute value $|x|$ of a real number $x$.

\subsection{A first example}

 While we do not consider negative $q$ in this paper (as is often the case), in this short section we show that the behaviour of the $L^q$-spectrum as $q \to -\infty$ for quasi-Bernoulli  measures on parabolic Cantor sets is very different from the hyperbolic case.  In some sense this motivates further study of the $L^q$-spectrum in the parabolic setting.  Given  a  Borel probability measure $\nu$ supported on compact set $X$ let
\begin{align*}
D_\nu(-\infty)  = \liminf_{r \to 0} \sup_{x \in X} \frac{\log \nu(B(x,r))}{\log r}.
\end{align*}
It can be shown that
\[
D_\nu(-\infty)  = \lim_{q \to -\infty} \frac{\underline{\tau}_\nu(q)}{-q}
\]
if $\underline{\tau}_\nu$ is appropriately extended to allow negative values of $q$, see \cite[Proposition 4.2]{boxdim}. The quantity $D_\nu(-\infty)$ is also known as the (lower) box dimension of $\nu$, recently introduced in \cite{boxdim}.    For quasi-Bernoulli measures $\nu$ supported on hyperbolic Cantor sets, it is straightforward to show that $D_\nu(-\infty)  < \infty$.  On the other hand, in the parabolic case $D_\nu(-\infty)$ is always infinite.

\begin{prop}
Suppose $\mathbb{P}$ is a measure on $\Sigma$ such that there exists $\theta \in (0,1)$ such that, for all sufficiently large $n$ and all $\i \in \mathcal{I}^n$, $\mathbb{P}([\i]) \leq \theta^n$.  In particular, quasi-Bernoulli measures satisfy this property.  If $\mu = \mathbb{P} \circ \Pi^{-1}$ is the associated measure on a parabolic Cantor set which  has a parabolic fixed point, then $D_\nu(-\infty) = \infty$.
\end{prop}

\begin{proof}
Suppose $h_i$   has a parabolic fixed point and let $\i_n = ii \cdots i \in \mathcal{I}^n$.   For large $n$ consider a ball $F \cap B(x,r) \subseteq h_{\i_n}(F)$ with  $r \approx | h_{\i_n}(F) |$.  Then
\[
\frac{\log \mu (B(x,r) )}{\log r} \gtrsim \frac{n \log \theta}{\log  | h_{\i_n}(F) |} \to \infty
\]
as $n \to \infty$ since $| h_{\i_n}(F) |$ must go to zero sub-exponentially since $h_i$ has a parabolic fixed point. This proves the claim. To see why quasi-Bernoulli measures satisfy this decay condition, choose an integer $N$ such that
\[
\max_{\j \in \mathcal{I}^N} \mathbb{P}([\j]) \leq 1/(3c). 
\]
where $c$ is the constant from the quasi-Bernoulli property. This can be done since $\mathbb{P}$ clearly has no atoms.  Then for arbitrary $\i \in \mathcal{I}^n$ we can write $\i = \j' \i'$ for some $\j'$ with length equal to $kN$ for some integer $k \geq 0$ and some $\i'$ with length less than $N$.  Then by the quasi-Bernoulli property
\[
\mathbb{P}([\i]) \leq \mathbb{P}([\j']) \leq c^k\left(\max_{\j \in \mathcal{I}^N} \mathbb{P}([\j]) \right)^k \leq 3^{-k} \leq 2^{-n/N}
\]
for  $n$ sufficiently large, proving the desired condition with $\theta = 2^{-1/N}$.
\end{proof}

Note that in the above proof we only used one side of the quasi-Bernoulli property and so the result also holds for submultiplicative measures, such as K\"aenm\"aki measures.

\subsection{Results and organisation}

We end this section by giving a heuristic version of our main results and describing the organisation of the paper. Roughly speaking, we will prove the following.

\begin{theorem*} Suppose $\mu$ is a weakly quasi-Bernoulli measure on a parabolic Cantor set or parabolic carpet. Suppose there is a large uniformly hyperbolic subsystem which has good distortion properties and such that the measure restricted to this system is quasi-Bernoulli. Then $\tau_\mu(q)$ is given by the minimal root of a natural pressure function.
\end{theorem*}

In the case of parabolic Cantor sets, a precise version of this result is given in Theorem \ref{cantorthm} and in the parabolic carpets case this is Theorem \ref{carpetthm}. 

The paper is organised as follows. In \S \ref{prel} we introduce the uniformly hyperbolic subsystem mentioned in the heuristic result above and discuss relevant distortion estimates as well as a further assumption on our measures. In \S \ref{top-cantor} we introduce the natural pressure function in the setting of parabolic Cantor sets and provide conditions under which the root of the pressure coincides with the critical exponent of a related `zeta function', which is the main technical result underpinning Theorem \ref{cantorthm}. In \S \ref{proof-cantor} we prove  Theorem \ref{cantorthm}. In \S \ref{top-carpet} we introduce the natural pressure function in the setting of parabolic carpets and provide conditions under which the root of the pressure coincides with the critical exponent of a related `zeta function', which is the main technical result underpinning Theorem \ref{carpetthm}. In \S \ref{proof-carpet} we prove  Theorem \ref{carpetthm}. In \S \ref{further} we discuss possible directions for future investigation.

\section{Preliminaries} \label{prel}

\subsection{Uniformly hyperbolic induced subsystem}

The following lemma establishes that the IFSs we consider necessarily  generate a uniform contraction.

\begin{lma} \label{existshyperbolic}
First consider an IFS $\{h_\i\}_{\i \in \I}$ generating  a parabolic Cantor set in $[0,1]$.  Then $\{h_\i\}_{\i \in \I^*}$ contains a map $h_{\i_0}$ such that $\| h_{\i_0}' \|_\infty <1$.  

Secondly, consider an  IFS $\{S_\i\}_{\i \in \I}$ generating  a parabolic carpet  in $[0,1]^2$.  Suppose that the parabolic carpet  is not contained in a vertical or horizontal line.  Then $\{S_\i\}_{\i \in \I^*}$ contains a map $S_{\i_0} = (f_{\i_0}, g_{\i_0})$ such that $\| f_{\i_0}' \|_\infty <1$ and  $\|g_{\i_0}' \|_\infty <1$.  Note that if the parabolic carpet  is  contained in a vertical or horizontal line, then it is really a parabolic Cantor set and we can apply the first result.
\end{lma}

\begin{proof}
Since by definition a parabolic Cantor set is not a single point,  there must exist $\i,\j \in \I^*$ such that the (unique)  fixed points of $h_\i$ and $h_\j$ are distinct.  By the chain rule, 
\[
\sup_{x \in [0,1]} |h_{\i\j}'(x)| =  \sup_{x \in [0,1]}  |h_\i'(h_\j(x)) | \cdot |h_\j'(x)| .
\]
By continuity of derivatives the supremum is attained at some $x \in [0,1]$.
Therefore, a necessary condition for  this supremum  to equal 1 is that $x$ is the fixed point of $h_\j$  and $h_\j(x)= x$ is the fixed point of $h_\i$ which is impossible.

If a parabolic carpet  is not contained in a vertical or horizontal line, then there must exist $\i,\j \in \I^*$ such that the (unique)  fixed points of $S_\i$ and $S_\j$ are not contained in a common vertical or horizontal line.  By the chain rule
\[
\sup_{x \in [0,1]} |f_{\i\j}'(x)| = |f_\i'(f_\j(x)) | \cdot |f_\j'(x)| .
\]
By continuity of derivatives the supremum is attained at some $x \in [0,1]$.
Therefore, a necessary condition for  this supremum to equal 1 is that $x$ is the fixed point of $f_\j$  and $f_\j(x)= x$ is the fixed point of $f_\i$ but this is not possible since the fixed points of $S_\i$ and $S_\j$ are not contained in a common vertical line.  A similar argument works for $g_{\i\j}$ and the result follows.
\end{proof}

Without loss of generality, we will assume $\i_0 \in \I$. A useful consequence of the existence of $\i_0$ from the previous lemma is that we have uniform contraction (hyperbolicity) on the subsystem of all words ending in $\i_0$.   That is,
\begin{equation}  \label{uniform}
\rho:=\sup_{\j \in \I^*}\|h'_{\j\i_0} \|_\infty <1.
\end{equation}
The analogous statements also hold for the families $\{f_i\}$ and $\{g_i\}$ and we assume the same $\rho$ works for all three families.   This is an immediate consequence of the chain rule since
\[
 \|h'_{\j\i_0} \|_\infty \leq \|h'_{\j} \|_\infty \|h'_{\i_0} \|_\infty  \leq \|h'_{\i_0} \|_\infty < 1.
\]

Once a hyperbolic index $\i_0 \in \I$ is understood, we define 
\begin{equation}\label{infty}
\I_\infty\vcentcolon=\{\j\i_0: \j \in (\I \setminus \{\i_0\})^*\} \cup \{\i_0\}.
\end{equation}
In order to study the parabolic Cantor sets and parabolic carpets, we will study the  `induced' uniformly hyperbolic infinite IFSs $\{h_\i\}_{\i \in \I_\infty}$ and $\{S_\i\}_{\i \in\I_\infty}$ respectively.

\subsection{Distortion estimates}

We have the  \emph{tempered distortion property} on the full alphabet $\I$, see e.g. \cite[Lemma 3]{gelfertrams}.

\begin{lma}\label{tdp}
There exists a sequence $p_n \to 0$ such that, for all $x,y \in [0,1]$ and $\i \in \I^n$,
\[
\left| \frac{h_\i'(x)}{h_\i'(y)}\right| \leq e^{np_n}.
\]
Moreover, without loss of generality we can assume that $e^{np_n}$ is increasing with $n$.  The same result holds for the families $\{ f_i\}_i$ and  $\{ g_i\}_i$ and we  choose $p_n$ to work for all three simultaneously. 
\end{lma}

\begin{proof}
  By the chain rule and A1'
\begin{align*}
\log \sup_{x,y \in [0,1]} \max_{\i \in \I^n}\left| \frac{h_\i'(x)}{h_\i'(y)}\right| \leq \sum_{k=0}^{n-1} \max_{i_1, \ldots, i_k}  \sup_{u,v \in h_{i_1 \ldots i_k}([0,1])} \max_{i \in \I}  \Big| \log |h_{i}'(u)| - \log |h_{i}'(v)|\Big| \lesssim  \sum_{k=0}^{n-1}  d_{k}^{\alpha_h}
\end{align*}
which, setting
\[
p_n = \frac{c}{n} \sum_{k=0}^{n-1} d_{k}^{\alpha_h},
\]
(for an appropriate constant $c$) completes the proof since $d_k \to 0$ as $k \to \infty$. In the above we used the fact that the derivatives $h_i'$ are non-vanishing and that $\log$ is Lipschitz on $[\epsilon_0, 1]$ for $\epsilon_0>0$.
\end{proof}

On the induced alphabet $\I_\infty$ we will require the bounded distortion property. The following lemma, which is similar to \cite[Lemma 3(b)]{bt}, gives conditions under which the induced IFSs satisfy bounded distortion.

\begin{lma}\label{bdp}
Suppose 
\begin{equation} \label{sum}
\sum_{\i \in \I_\infty} |h_\i([0,1])|^{\alpha_h}<\infty.
\end{equation}
Then there exists $1<C<\infty$ such that for all $\i \in \mathcal{I}_\infty^*$ and all $x,y \in [0,1]$,
\[
\frac{|h_\i'(x)|}{|h_\i'(y)|} \leq C.
\]
The analogous statement holds for the family $\{f_i\}$ under the assumption that $\sum_{\i \in \I_\infty} |f_\i([0,1])|^{\alpha_f}<\infty$  and for the family $\{g_i\}$ under the assumption that $\sum_{\i \in \I_\infty} |g_\i([0,1])|^{\alpha_g}<\infty$.  For these latter two conditions, abusing notation slightly, we only sum over $\i \in \mathcal{I}_\infty$ which give rise to \textbf{distinct} $f_\i$ and $g_\i$ respectively.  We assume the same $C$ works for all three families simultaneously. 
\end{lma}

\begin{proof}
By A1', there exists $C_h>0$ such that for all $i \in \I$, and $x, y \in [0,1]$, $|f_i'(x)-f_i'(y)| \leq C_h|x-y|^{\alpha_h}$. By A2' there exists $a>0$ such that $\inf_{i \in \I} \inf_{x \in [0,1]} |h_i'(x)| \geq a$. We begin by showing that there exists $C_0>0$ such that for all $\i \in \I_\infty$ and all $x,y \in [0,1]$,
\begin{equation} \label{level1v}
\left|\log \left(\frac{h_\i'(x)}{h_\i'(y)}\right)\right| \leq C_0 |x-y|^{\alpha_h}.
\end{equation}
Let $\i=i_1 \ldots i_n$ where each $i_j \in \I$. Then   by the chain rule
\begin{align}
\left|\log \left(\frac{h_\i'(x)}{h_\i'(y)}\right)\right| &\leq  \sum_{k=1}^n \left| \log\left( \frac{h_{i_k}'(h_{i_{k+1}\ldots i_n}(x))}{h_{i_k}'(h_{i_{k+1}\ldots i_n}(y))}\right)\right| \nonumber \\
&= \sum_{k=1}^n \log \left(1+ \frac{|h_{i_k}'(h_{i_{k+1}\ldots i_n}(x))-h_{i_k}'(h_{i_{k+1}\ldots i_n}(y))|}{|h_{i_k}'(h_{i_{k+1}\ldots i_n}(y))|}\right)  \nonumber\\
&\leq \sum_{k=1}^n \frac{|h_{i_k}'(h_{i_{k+1}\ldots i_n}(x))-h_{i_k}'(h_{i_{k+1}\ldots i_n}(y))|}{|h_{i_k}'(h_{i_{k+1}\ldots i_n}(y))|} \nonumber\\
&\leq \frac{C_h}{a}\sum_{k=1}^n |h_{i_{k+1}\ldots i_n}(x)-h_{i_{k+1}\ldots i_n}(y)|^{\alpha_h}. \label{sum1}
\end{align} 
The words $\{i_1\ldots i_n, i_2\ldots i_n, \ldots, i_n\}$ are distinct words in $\I_\infty$ therefore \eqref{sum1} implies that for all $\i \in \I_\infty$ and $x, y \in [0,1]$,
\begin{equation} \label{level1}
\left|\log \left(\frac{h_\i'(x)}{h_\i'(y)}\right)\right|  \leq \frac{C_h}{a}\sum_{\i \in \I_\infty} |h_\i([0,1])|^{\alpha_h}=:C'<\infty 
\end{equation}
by \eqref{sum}.

By \eqref{level1} and the mean value theorem, for all $\i \in \I_\infty$ and $x, y \in [0,1]$, 
\begin{equation} \label{lip}
|h_\i(x)-h_\i(y)| \leq e^{C'}|h_\i([0,1])||x-y|.
\end{equation}
To see this, fix $\i \in \I_\infty$ and $x,y \in [0,1]$. By the mean value theorem there exist $\theta, \theta'\in [0,1]$ such that $|h_\i([0,1])|=|h_\i'(\theta)|$ and $|h_\i(x)-h_\i(y)|=|h_\i'(\theta')||x-y|$. Thus by \eqref{level1} 
$$|h_\i(x)-h_\i(y)|=\frac{|h_\i'(\theta')|}{|h_\i'(\theta)|}|h_\i([0,1])||x-y| \leq e^{C'}|h_\i([0,1])| |x-y|,$$
 verifying \eqref{lip}. Plugging \eqref{lip} into \eqref{sum1} and using \eqref{sum} again yields \eqref{level1v}.

To complete the proof now let $\i=\i_1 \ldots \i_n \in \I_\infty^n$, where each $\i_j \in \I_\infty$. Then by the chain rule and \eqref{level1v},
\begin{align*}
\left|\log \left(\frac{h_\i'(x)}{h_\i'(y)}\right)\right| &\leq\sum_{k=1}^n\left| \log \left(\frac{h_{\i_k}'(h_{\i_{k+1}\ldots \i_n}(x))}{h_{\i_k}'(h_{\i_{k+1}\ldots i_n}(y))}\right)\right| \\
&\leq C_0 \sum_{k=1}^n |h_{\i_{k+1}\ldots \i_n}(x)-h_{\i_{k+1}\ldots i_n}(y)|^{\alpha_h}\\
&\leq C_0\sum_{k=0}^{n-1} \rho^{\alpha_hk} \leq \frac{C_0 }{1-\rho^{\alpha_h}}
\end{align*}
by \eqref{uniform}, completing the proof of the lemma.
\end{proof}

We briefly discuss condition  \eqref{sum}. In the case that a parabolic Cantor set satisfies A1' for $\alpha_h=1$ (i.e. each map $h_i$ is $C^{1+\textup{Lip}}$) then \eqref{sum} is immediate, indeed $\sum_{\i \in \I_\infty} |h_\i([0,1])|\leq 1$ by A3'. So if $\alpha_h=1$, Lemma \ref{bdp} holds without any additional assumption. Similarly, for any parabolic carpet which satisfies A1 for $\alpha_f=\alpha_g=1$, Lemma \ref{bdp} holds without any additional assumption.  This includes the case when all the maps in the IFS are $C^2$, see the central image in Figure 1.    Condition  \eqref{sum} also holds for many genuinely $C^{1+\alpha}$ IFSs including, for example, IFSs generated by the Manneville-Pomeau system, see Figure 1.  Given a parameter $\alpha \in (0,1)$, the Manneville-Pomeau map $f_\alpha: [0,1] \to [0,1]$ is given by
\begin{equation}\label{MP}
f_\alpha(x) =  \Big\{ 
    \begin{array}{cc}
      x+2^\alpha x^{1+\alpha} & 0 \leq x\leq 1/2 \\
      2x & 1/2< x\leq 1 
    \end{array}
\end{equation}
 The inverse branches of $f_\alpha$ form a parabolic $C^{1+\alpha}$  IFS on $[0,1]$  coded by $\{1,2\}$.  In particular, the right most branch is hyperbolic and choosing  $i_0$ corresponding to this map  we can verify condition  \eqref{sum} by a simple calculation.  Indeed, 
\[
\sum_{\i \in \I_\infty} |h_\i([0,1])|^{\alpha_h} = 1/2+  \sum_{n=1}^\infty \Big(f_\alpha^{-(n-1)}(1/2)-f_\alpha^{-n}(1/2)\Big)^\alpha \lesssim   \sum_{n=1}^\infty    n^{-(1+\alpha)} < \infty.
\]

\subsection{Almost quasi-Bernoulli property} \label{aqbp}

For our main results, we will require that our measure $\mathbb{P}$ is weakly quasi-Bernoulli, and that restricted to $\I_\infty^*$, $\mathbb{P}$  is quasi-Bernoulli, i.e. there exists a constant $c\geq 1$ such that for all $\i, \j \in \mathcal{I}_\infty^*$
\begin{equation} \label{qbp}
c^{-1} \leq \frac{\mathbb{P}([\i \j])}{\mathbb{P}([\i])\mathbb{P}([ \j])} \leq c.
\end{equation}

For concision, if $\mathbb{P}$ has the property that it is weakly quasi-Bernoulli on $\I$ and quasi-Bernoulli on $\I_\infty$ in the sense of \eqref{qbp}, we say that $\mathbb{P}$ and $\mathbb{P} \circ \Pi^{-1}$ are \emph{almost quasi-Bernoulli}. Clearly any quasi-Bernoulli measure (examples of which include Bernoulli measures and Gibbs measures for H\"older continuous potentials) is almost quasi-Bernoulli. However, the class of almost quasi-Bernoulli measures also include many other measures which are natural to study in the parabolic setting, which are \emph{not} quasi-Bernoulli e.g. the acip for the Manneville-Pomeau system \eqref{MP}.

\section{Topological pressure for parabolic Cantor sets} \label{top-cantor}

Given $s \in \mathbb{R}$ and $q \geq 0$, we define the \emph{topological pressure}   by
\[
P(s,q) =  \lim_{n \to \infty} \left(\sum_{\i \in \I^n} \mathbb{P}([\i])^q \| h_{\i}' \|_\infty^s\right)^{\frac{1}{n}}.
\]
We can show that for fixed $s,q$, $P(s,q)$ is well-defined.

\begin{lma}
Suppose $\mathbb{P}$ is weakly quasi-Bernoulli. For all $s \in \R$ and $q \geq 0$, $P(s,q)$ exists.
\end{lma}

\begin{proof}
We tackle the case that $s \geq 0$, the other case is similar. Let $N \in \N$ be sufficiently large that $\left(\sum_{\i \in \I^N} \Phi^{s,q}(\i)\right)^{\frac{1}{N}}< \liminf_{n \to \infty}\left(\sum_{\i \in \I^n} \mathbb{P}([\i])^q \| h_{\i}' \|_\infty^s\right)^{\frac{1}{n}}+\varepsilon$, noting that $N$ can be taken arbitrarily large. By \eqref{wqbp}
\begin{eqnarray*}
\left(\sum_{\i \in \I^{Nk}}\mathbb{P}([\i])^q \| h_{\i}' \|_\infty^s\right)^{\frac{1}{Nk}} &\leq& c_N^{\frac{q}{N}}\left(\sum_{\i \in \I^N} \mathbb{P}([\i])^q \| h_{\i}' \|_\infty^s\right)^{\frac{1}{N}} \\
&\leq& c_N^{\frac{q}{N}}\left(\liminf_{n \to \infty}\left(\sum_{\i \in \I^n}\mathbb{P}([\i])^q \| h_{\i}' \|_\infty^s\right)^{\frac{1}{n}}+\varepsilon\right)
\end{eqnarray*}
recalling that $c_N$ has the property that $c_N^{\frac{1}{N}} \to 1$ as $N \to \infty$.
Hence
$$\limsup_{k \to \infty} \left(\sum_{\i \in \I^{Nk}} \mathbb{P}([\i])^q \| h_{\i}' \|_\infty^s\right)^{\frac{1}{Nk}} \leq c_N^{\frac{q}{N}}\left(\liminf_{n \to \infty}\left(\sum_{\i \in \I^n} \mathbb{P}([\i])^q \| h_{\i}' \|_\infty^s\right)^{\frac{1}{n}}+\varepsilon\right).$$
The proof in the case where $s \geq 0$ is complete by observing that
$$\limsup_{k \to \infty} \left(\sum_{\i \in \I^{Nk}} \mathbb{P}([\i])^q \| h_{\i}' \|_\infty^s\right)^{\frac{1}{Nk}} =\limsup_{N \to \infty} \left(\sum_{\i \in \I^N} \mathbb{P}([\i])^q \| h_{\i}' \|_\infty^s\right)^{\frac{1}{N}}.$$
\end{proof}

It is straightforward to verify that, for a fixed $q \geq 0$, $P(s,q)$ is continuous and non-increasing in $s$.  In the hyperbolic setting the next step is to identify the `unique root' of the pressure, that is, a value $s_0=s_0(q)$ such that $P(s_0, q) = 1$.  However, this does not have to exist in the parabolic setting.  Indeed, for a parabolic Cantor set the pressure is easily seen to satisfy $P(s,0) \geq 1$ for all $s \in \mathbb{R}$.  This can be shown directly by just considering iterates of a single map with a parabolic fixed point.  In this case: if there is a root, then there are infinitely many roots. 

Since Lemma \ref{tdp} ensures that for all $\i \in \mathcal{I}^n$
\[
e^{-np_n} \| h_{\i}' \|_\infty \leq \inf_{x \in [0,1]} |h_\i'(x) | \leq  \| h_{\i}' \|_\infty
\]
we may replace $ \| h_{\i}' \|_\infty$ by $\inf_{x \in [0,1]} |h_\i'(x) |$ (or any intermediate value) in the definition of pressure without changing the value of the limit.  For example, we may use $|h_\i([0,1])|$, which is the length of the $n$th level construction interval associated with $\i$.


\begin{lma}
Given any measure on a  parabolic Cantor set, $P(1,q) \leq 1$ for all $q \geq 0$.
\end{lma}
\begin{proof}
Using A3', for all $n$ the sets 
\[
\{ h_\i ((0,1)) : \i \in \I^{n}\}
\]
are pairwise disjoint and contained in $(0,1)$.    Therefore, for all $q \geq 0$
\[
P(1,q) \leq P(1,0) =  \lim_{n \to \infty} \left(\sum_{\i \in \I^n}   |h_\i([0,1])| \right)^{\frac{1}{n}} \leq 1
\]
as required.
\end{proof}

The previous lemma allows us to define a function $\gamma: [0,\infty) \to \mathbb{R}$ by
$$\gamma(q):=\inf\{s>0: P(s,q) \leq 1\}.$$
We will now provide criteria under which $\gamma(q)$ may also be realised as  the critical exponent $\delta(q)$ of the zeta function
$$\zeta(s,q):= \sum_{n=1}^\infty \sum_{\i \in \I^n} \mathbb{P}([\i])^q \| h_{\i}' \|_\infty^s.$$

By Lemma \ref{existshyperbolic} (see also \eqref{uniform}),  there exists $\i_0 \in \I^{*}$, such that $h_{\i_0}$ is a uniform contraction (hyperbolic) and $\{h_{\j\i_0}: \j \in \I^*\}$ are a set of uniform contractions \eqref{uniform}.
 Since $\delta_\I= \delta_{\I^{k}}$ and $\gamma_{\I}=\gamma_{\I^{k}}$ for all $k$, without loss of generality we can assume $\i_0 \in \I$.

Recall that we defined 
\[
\I_\infty\vcentcolon=\{\j\i_0: \j \in (\I \setminus \{\i_0\})^*\} \cup \{\i_0\}.
\]

\begin{prop}\label{crit2}
For an almost quasi-Bernoulli measure on a parabolic Cantor set which satisfies $\sum_{\i \in \I_\infty} |h_\i([0,1])|^{\alpha_h}<\infty$, we have $\gamma(q) = \delta(q)$ for all $q \geq 0$. In particular, this holds for all $C^2$ or $C^{1+\textup{Lip}}$ parabolic Cantor sets.
\end{prop}

Throughout the rest of this section, fix $q \geq 0$.  Given a digit set $\J \subset \mathcal{I}^*$   let $\gamma_\J$ denote the minimal root of the pressure $P_\J(s,q)=\lim_{n \to \infty}\left(\sum_{\i \in \J^n}  \mathbb{P}([\i])^q \| h_{\i}' \|_\infty^s\right)^{\frac{1}{n}} $  and $\delta_\J$ denote the critical exponent of the zeta function $\zeta(s,q)=\sum_{n=1}^\infty\sum_{\i \in \J^n}  \mathbb{P}([\i])^q \| h_{\i}' \|_\infty^s$.  Observe that $\gamma_\J\leq \delta_\J$ for any finite or countable set $\J \subset \I^*$.

We denote
\[
\I_N\vcentcolon= \I_\infty \cap \left(\bigcup_{m=1}^N \I^m\right).
\]

The proof of Proposition \ref{crit2} will follow from four lemmas where we show that:
\[
\mbox{\textbf{(i)}}   \ \ \delta_\I=\delta_{\I_\infty},  \ \ \ \ \ \  
\mbox{\textbf{(ii)}}  \ \ \delta_{\I_\infty}=\gamma_{\I_\infty},  \ \ \ \ \ \ 
\mbox{\textbf{(iii)}} \ \ \gamma_{\I_{\infty}}=\sup_N \gamma_{\I_N}, \ \ \ \ \ \ 
\mbox{\textbf{(iv)}} \ \sup_N \gamma_{\I_N} =\gamma_\I. \ \ \ \ \ \ 
\]

 In the proofs that follow we'll assume that $\delta_\I<\infty$, but one can easily see that the proofs also imply $\gamma_\I=\infty$ if $\delta_\I=\infty$.

\begin{lma}[Proof of \textbf{(ii)}]\label{ii2}
$\delta_{\I_\infty}=\gamma_{\I_\infty}$.
\end{lma}

\begin{proof}
 From the root test, $P_{\I_\infty}(s,q) \geq 1$, for $s < \delta_{\I_\infty}$, and $P_{\I_\infty}(s,q) \leq 1$, for $s > \delta_{\I_\infty}$. Therefore, $\delta_{\I_\infty}$ is a root of $P_{\I_\infty}(\cdot, q)$. By \eqref{uniform}, $\| h_{\i}' \|_\infty \leq \rho^{n}$ for any $\i \in \I_\infty^n$. This implies that $P_{\I_\infty}( \cdot, q)$ has a unique root. Indeed, 
\[
\sum_{\i \in \I_\infty^n}\mathbb{P}([\i])^q \| h_{\i}' \|_\infty^{\gamma_{\I_\infty}+\epsilon} \leq \rho^{n\epsilon} \sum_{\i \in \I_\infty^n} \mathbb{P}([\i])^q \| h_{\i}' \|_\infty^{\gamma_{\I_\infty}}.
\]
In particular, $P_{\I_\infty}(\gamma_{\I_\infty}+\epsilon, q) \leq \rho^{\epsilon}\, P_{\I_\infty}(\gamma_{\I_\infty}, q)=\rho^\epsilon<1$. Therefore $\delta_{\I_\infty}$ is the unique root of $P_{\I_\infty}( \cdot, q)$, that is, $\gamma_{\I_\infty}=\delta_{\I_\infty}$ as required.
\end{proof}

\begin{lma}[Proof of  \textbf{(i)}] \label{i2}
$\delta_\I=\delta_{\I_\infty}$.
\end{lma}

\begin{proof}
 Since $\I_\infty^* \subseteq \I^*$, we have $\delta_{\I_\infty} \leq \delta_\I$. We claim that there exists a sequence $D_n$ such that $\lim_{n \to \infty} D_n^{\frac{1}{n}}=1$ and
\[
\mathbb{P}([\i])^q \| h_{\i}' \|_\infty^{s} \leq D_n\mathbb{P}([\i \i_0])^q \| h_{\i\i_0}' \|_\infty^{s}
\]
 for all $\i\in \I^n$. Indeed, for any $\i \in \I^n$,
\[
\frac{\mathbb{P}([\i \i_0])^q \| h_{\i\i_0}' \|_\infty^{s}}{\mathbb{P}([\i])^q \| h_{\i}' \|_\infty^{s}} \gtrsim  c_n^{-q}\mathbb{P}([\i_0])^q e^{-|s|np_n}\| h_{\i_0}' \|_\infty^{s}
\]
by the weak quasi Bernoulli property \eqref{wqbp} and the chain rule followed by the tempered distortion property (Lemma \ref{tdp}). This proves the claim. Therefore 
\begin{align}
\zeta_\I(s,q) &= \sum_{n=1}^{\infty} \sum_{\substack{\i \in \I^n \\ i_n=\i_0}}\mathbb{P}([\i])^q \| h_{\i}' \|_\infty^{s}+\sum_{n=1}^{\infty} \sum_{\substack{\i \in \I^n \\ i_n\neq \i_0}} \mathbb{P}([\i])^q \| h_{\i}' \|_\infty^{s} \nonumber\\
&\leq \zeta_{\I_\infty}(s,q) +    \sum_{n=1}^{\infty} D_n\sum_{\substack{\i \in \I^n \\ i_n\neq \i_0}} \mathbb{P}([\i \i_0])^q \| h_{\i\i_0}' \|_\infty^{s}. \label{zeta ineq2}
\end{align}
If $s> \delta_{\I_\infty}$,  then $ \zeta_{\I_\infty}(s,q)<\infty$ and (as seen in proof of Lemma \ref{ii2})
$$ \lim_{n \to \infty}\left(\sum_{\substack{\i \in \I^n \\ i_n\neq \i_0}} \mathbb{P}([\i \i_0])^q \| h_{\i\i_0}' \|_\infty^{s}\right)^{\frac{1}{n}} \leq P_{\I_\infty}(s,q)<1.$$
 Since $\lim_{n \to\infty} D_n^{\frac{1}{n}}=1$,  the right hand side of \eqref{zeta ineq2} is summable for $s>\delta_{\I_\infty}$ by the root test.  In particular $\delta_\I \leq \delta_{\I_\infty}$.
\end{proof}

\begin{lma}[Proof of  \textbf{(iii)}] \label{iii2}
$\gamma_{\I_{\infty}}=\sup_N \gamma_{\I_N}.$
\end{lma}

\begin{proof}
Since $\I_N \subseteq \I_\infty$, clearly $\sup_N \gamma_{\I_N} \leq \gamma_{\I_\infty}$. Therefore it is sufficient to show that for all $\epsilon>0$ there exists $N \in \N$ such that $\gamma_{\I_N}>\gamma_{\I_\infty}-\epsilon$. Recall that by Lemma \ref{bdp} the family $\{h_\i\}_{\i \in \I_\infty}$ satisfies the bounded distortion property, which will be crucial for this proof. 

 Fix $\epsilon>0$ and write $s=\gamma_{\I_\infty}-\epsilon$. Let $1<\lambda<P_{\I_\infty}(\gamma_{\I_\infty}-\epsilon,q)$ and choose $n$ sufficiently large that 
$$\left(\sum_{\i \in \I_\infty^n} \mathbb{P}([\i])^q \| h_{\i}' \|_\infty^{s}\right)^{\frac{1}{n}}>\frac{\lambda+P_{\I_\infty}(\gamma_{\I_\infty}-\epsilon,q)}{2}$$ and $$\left(\frac{1}{c^{q}C^{|s|} }\right)^{\frac{1}{n}} >\frac{1}{\lambda}$$
where $C$ is the constant from Lemma \ref{bdp} and $c$ is the constant from \eqref{qbp}.  Since
\[
\lim_{N \to \infty} \left(\sum_{\i \in \I_N^n}\mathbb{P}([\i])^q \| h_{\i}' \|_\infty^{s}\right)^{\frac{1}{n}}=\left(\sum_{\i \in \I_\infty^n}\mathbb{P}([\i])^q \| h_{\i}' \|_\infty^{s}\right)^{\frac{1}{n}}>\frac{\lambda+P_{\I_\infty}(\gamma_{\I_\infty}-\epsilon,q)}{2}>\lambda,
\]
we can choose $N$ sufficiently large that
$$ \left(\sum_{\i \in \I_N^n}\mathbb{P}([\i])^q \| h_{\i}' \|_\infty^{s}\right)^{\frac{1}{n}}>\lambda.$$
For all   $\i_1, \ldots, \i_k \in  \I_N^n$, Lemma \ref{bdp} and \eqref{qbp} guarantee
\[
\mathbb{P}([\i_1 \ldots \i_k])^q \| h_{\i_1 \ldots \i_k}' \|_\infty^{s} \geq c^{-qk}C^{-|s|k}\mathbb{P}([\i_1 ])^q \| h_{\i_1}' \|_\infty^{s} \cdots \mathbb{P}([  \i_k])^q \| h_{  \i_k}' \|_\infty^{s}.
\]
 In particular,
\[
P_{\I_N}(s,q)=\lim_{k \to \infty}\left(\sum_{\i \in \I_N^{nk}} \mathbb{P}([\i])^q \| h_{\i}' \|_\infty^{s} \right)^{\frac{1}{nk}} 
\geq\left(\frac{1}{c^{q}C^{|s|} }\right)^{\frac{1}{n}} \left(\sum_{\i \in \I_N^n}\mathbb{P}([\i])^q \| h_{\i}' \|_\infty^{s}\right)^{\frac{1}{n}}>\frac{1}{\lambda} \cdot \lambda=1.
\]
In particular $\gamma_{\I_N}>\gamma_{\I_\infty}-\epsilon$, completing the proof of \textbf{(iii)}. \end{proof}

\begin{lma}[Proof of \textbf{(iv)}]
$\sup_N \gamma_{\I_N}=\gamma_\I$.
\end{lma}

\begin{proof}
Note that since $\gamma_\I \leq \delta_\I= \sup_N \gamma_{\I_N}$ by  \textbf{(i-iii)} above, it is sufficient to prove that $\sup_N \gamma_{\I_N} \leq \gamma_\I$.  Let $N \in \N$. Suppose $\gamma_\mathcal{I} \geq 0$.  Then by submultiplicativity and \eqref{wqbp}
\begin{align*}
\left(\sum_{\i \in \I_N^{k\ell}} \mathbb{P}([\i])^q \| h_{\i}' \|_\infty^ {\gamma_\I}\right)^{\frac{1}{Nk\ell}} \leq \left(\sum_{\i \in \I^{Nk\ell}} \mathbb{P}([\i])^q \| h_{\i}' \|_\infty^ {\gamma_\I}\right)^{\frac{1}{Nk\ell}}&\leq \left(\left(\sum_{\i \in \I^{N\ell}} \mathbb{P}([\i])^q \| h_{\i}' \|_\infty^ {\gamma_\I}\right)^kc_{N \ell}^{qk}\right)^{\frac{1}{Nk\ell}}\\
&= c_{N\ell}^{\frac{q}{N\ell}}\left(\sum_{\i \in \I^{N\ell}} \mathbb{P}([\i])^q \| h_{\i}' \|_\infty^ {\gamma_\I}\right)^{\frac{1}{N\ell}}.
\end{align*}
Since $\lim_{\ell \to \infty} c_{N\ell}^{1/N\ell}=1$,
$$P_{\I_N}(\gamma_\I)^{\frac{1}{N}}=\lim_{\ell \to \infty}\left(\sum_{\i \in \I_N^{k\ell}}\mathbb{P}([\i])^q \| h_{\i}' \|_\infty^ {\gamma_\I}\right)^{\frac{1}{Nk\ell}}\leq \lim_{\ell \to \infty}\left(\sum_{\i \in \I^{N\ell}} \mathbb{P}([\i])^q \| h_{\i}' \|_\infty^ {\gamma_\I}\right)^{\frac{1}{N\ell}}=P_\I(\gamma_\I)=1.$$
In particular $P_{\I_N}(\gamma_\I) \leq 1$ and so  $\gamma_{\I_N} \leq \gamma_\I$, completing the proof in the case $\gamma_\mathcal{I} \geq 0$.  If $\gamma_\mathcal{I} < 0$, then we may replace 
\[
\mathbb{P}([\i])^q \| h_{\i}' \|_\infty^ {\gamma_\I} = \mathbb{P}([\i])^q \sup_{x \in [0,1]} | h_{\i}' (x)|^{\gamma_\mathcal{I}}
\]
by 
\[
 \mathbb{P}([\i])^q \inf_{x \in [0,1]} | h_{\i}' (x)|^{\gamma_\mathcal{I}}
\]
in which case we recover submultiplicativity.  Moreover, by Lemma \ref{tdp} this does not change the value of the pressure $P_{\I_N}(\gamma_\I)$ or $P_\I(\gamma_\I)$ and so the proof goes through as above.
\end{proof}

\section{$L^q$-spectrum of almost quasi-Bernoulli measures on parabolic Cantor sets} \label{proof-cantor}

We are now ready to state our  main contribution to the dimension theory of parabolic Cantor sets. This is an important stepping stone towards our main result, Theorem \ref{carpetthm}, which concerns parabolic carpets.

\begin{thm} \label{cantorthm}
If  $\mu$ is an almost quasi-Bernoulli measure on a parabolic Cantor set $F$ for which $\sum_{\i \in \I_\infty} |h_\i([0,1])|^{\alpha_h}<\infty$, then 
$$\tau_\mu(q) = \gamma(q)$$
for all $q \geq 0$.  Moreover, $\bd F = \hd F = \gamma(0)$.   In particular, these results  hold for all $C^2$ or $C^{1+\textup{Lip}}$ parabolic Cantor sets.
\end{thm}

\subsection{Proof of upper bound for $\tau_\mu$ in Theorem \ref{cantorthm}}
Let $q \geq 0$ and $s>\gamma(q)$. Given $0<\delta<1$ define the $\delta$-stopping $\mathcal{S}_\delta:=\{\i \in \I^*: | h_{\i}([0,1]) | \leq \delta< | h_{\i^-}([0,1])| \}$ where $\i^-$ is obtained from $\i$ by removing the final symbol in $\I$. Note that (using A2)  $| h_{\i}([0,1]) |\approx \delta$ for $\i \in \mathcal{S}_\delta$. Then
\begin{eqnarray*}
\delta^{s}D^q_\delta(\mu) &\approx& \delta^{s}\sum_{\i \in \mathcal{S}_\delta}\mathbb{P}([\i])^q\\
&\approx& \sum_{\i \in \mathcal{S}_\delta}\mathbb{P}([\i])^q| h_{\i}([0,1]) |^{s}\\
&\leq& \sum_{n=1}^\infty \sum_{\i \in \mathcal{I}^n}\mathbb{P}([\i])^q\| h_{\i}' \|_\infty^{s} < \infty
\end{eqnarray*}
  by Proposition \ref{crit2}.

\subsection{Proof of lower bound for $\tau_\mu$ in Theorem \ref{cantorthm}}

 Let $q \geq 0$ and $s<\gamma(q)$. Choose $N$ sufficiently large such that $s<\gamma_N<\gamma(q)$ where $\gamma_N = \gamma_{\mathcal{I}_N}$ is the root of the pressure (for our fixed $q$) associated with the hyperbolic subsystem $\mathcal{I}_N$. We can choose such an $N$ by Proposition \ref{crit2}. Given $0<\delta<1$ define $\mathcal{S}^N_\delta:=\{\i \in (\I_N)^*:  | h_{\i}([0,1]) | \leq \delta< | h_{\i^-}([0,1])| \}$.  This time   $\i^-$ is obtained from $\i$ by removing the final symbol in $\I_N$. Note that $| h_{\i}([0,1]) | \approx_N \delta$ for $\i \in \mathcal{S}^N_\delta$.  Then
\begin{eqnarray*}
\delta^{s}D^q_\delta(\mu) &\approx_N& \delta^{s}\sum_{\i \in \mathcal{S}^N_\delta}\mathbb{P}([\i])^q\\
&\approx_N& \sum_{\i \in \mathcal{S}^N_\delta} \mathbb{P}([\i])^q| h_{\i}([0,1]) |^{s}\\
&\approx& \sum_{\i \in \mathcal{S}^N_\delta}c^qC^{3|s|}\mathbb{P}([\i])^q| h_{\i}([0,1]) |^{s}
\end{eqnarray*}
where $c$ is the constant from \eqref{qbp} and $C$ is the constant from Lemma \ref{bdp}, recalling that we may apply bounded distortion to the subsystem $\mathcal{I}_N$. To complete the proof we show the final term is bounded (strictly) below by 1.  Suppose to the contrary that 
\begin{equation} \label{contra1}
\sum_{\i \in \mathcal{S}^N_\delta}c^qC^{3|s|}\mathbb{P}([\i])^q| h_{\i}([0,1]) |^{s} \leq 1.
\end{equation}
Therefore, using \eqref{qbp} and Lemma \ref{bdp} (three times), for all $\i \in (\mathcal{I}_N)^*$, 
\begin{align} \label{iteratecantor}
 \sum_{\j \in \mathcal{S}^N_\delta} c^qC^{3|s|}\mathbb{P}([\i\j])^q| h_{\i\j}([0,1]) |^{s} &\leq   \sum_{\j \in \mathcal{S}^N_\delta}\left(c^qC^{3|s|}\mathbb{P}([\i])^q| h_{\i}([0,1]) |^{s} \right) \left(c^qC^{3|s|}\mathbb{P}([\j])^q| h_{\j}([0,1]) |^{s}\right) \nonumber \\
&\leq   c^qC^{3|s|}\mathbb{P}([\i])^q| h_{\i}([0,1]) |^{s} . 
\end{align}
Let $k \geq 1$ be a very large integer. We approximate words of length $k$ by compositions of words in $\mathcal{S}^N_\delta$ by defining
\[
\mathcal{S}^{N,k}_\delta = \{ \i_1 \cdots \i_m : \forall l=1, \dots, m,  \,  \i_l \in \mathcal{S}^N_\delta, \  \exists \i_{m+1} \in \mathcal{S}^N_\delta \ \text{s.t.} \ |\i_1 \cdots \i_m| \leq k < |\i_1 \cdots \i_m\i_{m+1}| \}.
\]
Then, by iteratively applying \eqref{iteratecantor}, and then \eqref{contra1}
\begin{equation} \label{contra}
 \sum_{\i \in \mathcal{S}^{N,k}_\delta}c^qC^{3|s|}\mathbb{P}([\i])^q| h_{\i}([0,1]) |^{s} \leq 1.
\end{equation}
For all $\i \in (\mathcal{I}_N)^k$ we can write $\i = \i_1 \i_2$ for some $\i_1 \in \mathcal{S}^{N,k}_\delta$ and $\i_2 \in \mathcal{I}_N^*$ with $|\i_2| \lesssim_\delta 1$.  Therefore
\[
 \sum_{\i \in  (\mathcal{I}_N)^k}c^qC^{3|s|}\mathbb{P}([\i])^q| h_{\i}([0,1]) |^{s} \lesssim_\delta   \sum_{\i \in  \mathcal{S}^{N,k}_\delta}c^qC^{3|s|}\mathbb{P}([\i])^q| h_{\i}([0,1]) |^{s} \leq 1
\]
by \eqref{contra}
which proves (letting $k \to \infty$ while keeping $\delta$ fixed) that $s \geq s_N$, a contradiction.

\subsection{Theorem \ref{cantorthm}: dimensions of parabolic Cantor sets}

The fact that $\bd F = \gamma(0)$ is immediate.  The fact that this value also coincides with the Hausdorff dimension of $F$ follows from Proposition \ref{crit2} since, writing $F_N \subseteq F$ for the attractor of the hyperbolic IFS  $\{ h_i \}_{i \in \mathcal{I}_N}$,
\[
\hd F \geq \sup_N \hd F_{N} = \sup_N \bd F_{N} =  \sup_N \gamma_{\mathcal{I}_N}(0) = \gamma(0).
\]
Here we have used that the Hausdorff and box dimensions of a hyperbolic Cantor set coincide. 

\section{Topological pressure for parabolic carpets} \label{top-carpet}

\subsection{Singular values and singular value function}

For $\i \in \I^*$ define 
$$\o_1(\i):= \max_{x \in [0,1]}\max \{|f_\i'(x)|, |g_\i'(x)|\} \;\;\; \textnormal{and} \;\;\; \o_2(\i):= \max_{x \in [0,1]}\min \{|f_\i'(x)|, |g_\i'(x)|\}.$$
Note that if $\i \in \I^n$ then
\begin{equation}\label{td1}
\u_1(\i):=\min_{x \in [0,1]} \max \{|f_\i'(x)|, |g_\i'(x)|\}  \leq \o_1(\i) \leq e^{np_n} \min_{x \in [0,1]} \max \{|f_\i'(x)|, |g_\i'(x)|\}=e^{np_n}\u_1(\i) 
\end{equation} and 
\begin{equation} \label{td2}
\u_2(\i):=\min_{x \in [0,1]} \min \{|f_\i'(x)|, |g_\i'(x)|\}  \leq \o_2(\i) \leq e^{np_n} \min_{x \in [0,1]} \min \{|f_\i'(x)|, |g_\i'(x)|\}=e^{np_n}\u_2(\i). \end{equation}

For $q \geq 0$, let $t(\i, q):= \tau_{\pi_\i \mu}(q)$ where $\pi_\i$ is the projection in the direction of the longer side of $S_\i([0,1]^2)$.  Crucially by Theorem \ref{cantorthm} $t(\i, q)$ exists.  Let $t_f(q)=\tau_{\pi_1 \mu}(q)$ where $\pi_1$ corresponds to projection to the $x$ axis and $t_g(q)=\tau_{\pi_2 \mu}(q)$ where $\pi_2$ corresponds to projection to the $y$ axis.  So for all $\i \in \I^*$, $t(\i,q) \in \{t_f(q),t_g(q)\}$.  

For $s \in \mathbb{R}$ and $q \geq 0$, the constant 
\begin{equation} \label{sillyconstant}
T(s,q) = |s|+2 \max\{ |t_f(q)|,|t_g(q)|\} \geq 0
\end{equation}
will naturally appear in several places, especially when we apply Lemma  \ref{tdp}.

\begin{lma} \label{mult}
Fix $q \geq 0$ and $s \in \mathbb{R}$.  Suppose $\mathbb{P}$ is weakly quasi-Bernoulli and let $\Phi^{s,q}: \I^* \to \R$ be any singular value function of the form
\begin{equation} \label{sing}
\Phi^{s,q}(\i)=\mathbb{P}([\i])^q \alpha_1(\i)^{t(\i,q)} \alpha_2(\i)^{s-t(\i,q)}
\end{equation}
 where $\u_1(\i) \leq \alpha_1(\i) \leq \o_1(\i)$ and $\u_2(\i) \leq \alpha_2(\i)\leq \o_2(\i)$. Then there exists an increasing sequence of positive numbers $(C_n)_{n \in \N}$  such that for any $\i_1 \in \I^{n_1}, \ldots , i_k \in \I^{n_k}$:
\begin{enumerate}[(i)]
\item if $s<t_f(q)+t_g(q)$ then 
$$\Phi^{s,q}(\i_1 \ldots \i_k) \leq C_{n_1}\cdots C_{n_k}\Phi^{s,q}(\i_1)\cdots\Phi^{s,q}(\i_k),$$
\item if $s=t_f(q)+t_g(q)$ then 
$$C_{n_1}^{-1}\cdots C_{n_k}^{-1}\Phi^{s,q}(\i_1)\cdots\Phi^{s,q}(\i_k)\leq \Phi^{s,q}(\i_1 \ldots \i_k) \leq C_{n_1}\cdots C_{n_k}\Phi^{s,q}(\i_1)\cdots\Phi^{s,q}(\i_k),$$
\item if $s>t_f(q)+t_g(q)$ then 
$$ \Phi^{s,q}(\i_1 \ldots \i_k)\geq C_{n_1}^{-1}\cdots C_{n_k}^{-1}\Phi^{s,q}(\i_1)\cdots\Phi^{s,q}(\i_k).$$ 
\end{enumerate}
Moreover, $\lim_{n \to \infty} C_n^{\frac{1}{n}}=1$.
\end{lma}

\begin{proof}
We will prove this for the singular value function $\Phi^{s,q}(\i)=\mathbb{P}([\i])^q \alpha_1(\i)^{t(\i,q)} \alpha_2(\i)^{s-t(\i,q)}$ where $\alpha_1(\i)$ denotes the length of the longer side of $S_\i([0,1]^2)$ and $\alpha_2(\i)$ denotes the length of the shorter side of  $S_\i([0,1]^2)$, noting that these definitions satisfy the assumptions of the lemma. The result for other singular value functions will then follow from \eqref{td1} and \eqref{td2}. 

For $\i \in \I^*$ let $w(\i)$ denote the width of $S_{\i}([0,1]^2)$ and $h(\i)$ denote the height of $S_{\i}([0,1]^2)$. Without loss of generality we can assume that $w(\i_1 \ldots \i_k)>h(\i_1 \ldots \i_k)$, so that $t(\i_1 \ldots \i_k,q)=t_f(q)$.

For $s \leq t_f(q)+t_g(q)$, notice that by Lemma \ref{tdp} and \eqref{qbp}
\begin{eqnarray*}
&\,& \hspace{-1cm} \frac{\Phi^{s,q}(\i_1 \ldots \i_k)}{\Phi^{s,q}(\i_1)\cdots \Phi^{s,q}(\i_k)}\\
& =&\frac{\mathbb{P}([\i_1 \ldots \i_k])^qw(\i_1 \ldots \i_k)^{t_f(q)}h(\i_1 \ldots \i_k)^{s-t_f(q)}}{\mathbb{P}([\i_1])^q\alpha_1(\i_1)^{t(\i_1,q)}\alpha_2(\i_1)^{s-t(\i_1,q)}\cdots \mathbb{P}([\i_k])^q\alpha_1(\i_k)^{t(\i_k,q)}\alpha_2(\i_k)^{s-t(\i_k,q)}}\\
&\leq& c_{n_1}^{q}\cdots c_{n_k}^{q}e^{T(s,q)(n_1p_{n_1}+\cdots+n_kp_{n_k})} \frac{w(\i_1)^{t_f(q)}h(\i_1)^{s-t_f(q)}\cdots w(\i_k)^{t_f(q)}h(\i_k)^{s-t_f(q)}}{\alpha_1(\i_1)^{t(\i_1,q)}\alpha_2(\i_1)^{s-t(\i_1,q)}\cdots \alpha_1(\i_k)^{t(\i_k,q)}\alpha_2(\i_k)^{s-t(\i_k,q)}}\\
&=& c_{n_1}^{q}\cdots c_{n_k}^{q}e^{T(s,q)(n_1p_{n_1}+\cdots+n_kp_{n_k})} \prod_{\substack{1 \leq j \leq k:\\ w(\i_j)<h(\i_j)}} \left(\frac{h(\i_j)}{w(\i_j)}\right)^{s-t_f(q)-t_g(q)}\\
&\leq&  c_{n_1}^{q}\cdots c_{n_k}^{q} e^{T(s,q)(n_1p_{n_1}+\cdots+n_kp_{n_k})}
\end{eqnarray*}
since $s \leq t_f(q)+t_g(q)$. 

Similarly, for $s \geq t_f(q)+t_g(q)$
\begin{eqnarray*}
&\,& \hspace{-1cm}  \frac{\Phi^{s,q}(\i_1 \ldots \i_k)}{\Phi^{s,q}(\i_1)\cdots \Phi^{s,q}(\i_k)}\\
& =&\frac{\mathbb{P}([\i_1 \ldots \i_k])^qw(\i_1 \ldots \i_k)^{t_f(q)}h(\i_1 \ldots \i_k)^{s-t_f(q)}}{\mathbb{P}([\i_1])^q\alpha_1(\i_1)^{t(\i_1,q)}\alpha_2(\i_1)^{s-t(\i_1,q)}\cdots \mathbb{P}([\i_k])^q\alpha_1(\i_k)^{t(\i_k,q)}\alpha_2(\i_k)^{s-t(\i_k,q)}}\\
&\geq&c_{n_1}^{-q}\cdots c_{n_k}^{-q} e^{-T(s,q)(n_1p_{n_1}+\cdots+n_kp_{n_k})} \frac{w(\i_1)^{t_f(q)}h(\i_1)^{s-t_f(q)}\cdots w(\i_k)^{t_f(q)}h(\i_k)^{s-t_f(q)}}{\alpha_1(\i_1)^{t(\i_1,q)}\alpha_2(\i_1)^{s-t(\i_1,q)}\cdots \alpha_1(\i_k)^{t(\i_k,q)}\alpha_2(\i_k)^{s-t(\i_k,q)}}\\
&=&c_{n_1}^{-q}\cdots c_{n_k}^{-q} e^{-T(s,q)(n_1p_{n_1}+\cdots+n_kp_{n_k})}\prod_{\substack{1 \leq j \leq k:\\ w(\i_j)<h(\i_j)}} \left(\frac{h(\i_j)}{w(\i_j)}\right)^{s-t_f(q)-t_g(q)}\\
&\geq & c_{n_1}^{-q}\cdots c_{n_k}^{-q} e^{-T(s,q)(n_1p_{n_1}+\cdots+n_kp_{n_k})}
\end{eqnarray*}
since $s \geq t_f(q)+t_g(q)$. Taking $C_n:=c_n^qe^{T(s,q)np_n}$ gives the result, noting that $\lim_{n \to \infty} C_n^{\frac{1}{n}}=1$ since $p_n \to 0$ and $\lim_{n \to \infty}c_n^{\frac{1}{n}} = 1$.
\end{proof}

\subsection{Pressure: existence and basic properties}

Define the pressure function $P:\mathbb{R} \times [0,\infty)  \to [0,\infty]$ by
$$P(s,q):= \lim_{n \to \infty} \left(\sum_{\i \in \I^n}  \mathbb{P}([\i])^q \o_1(\i)^{t(\i,q)} \o_2(\i)^{s-t(\i,q)}\right)^{\frac{1}{n}}.$$

\begin{lma}
For all $s \in \R$ and $q \geq 0$ the pressure $P(s,q)$ exists.
\end{lma}

\begin{proof}
We begin by fixing $s \geq t_f(q)+t_g(q)$ (the case where $s<t_f(q)+t_g(q)$ will be similar) and let $\Phi^{s,q}$ denote the singular value function. Fix $\varepsilon>0$. Let $N \in \N$ be sufficiently large that $\left(\sum_{\i \in \I^N} \Phi^{s,q}(\i)\right)^{\frac{1}{N}}> \limsup_{n \to \infty}\left(\sum_{\i \in \I^n} \Phi^{s,q}(\i)\right)^{\frac{1}{n}}-\varepsilon$, noting that $N$ can be taken arbitrarily large. By Lemma \ref{mult}
\begin{eqnarray*}
\left(\sum_{\i \in \I^{Nk}} \Phi^{s,q}(\i)\right)^{\frac{1}{Nk}} &\geq& C_N^{\frac{1}{N}}\left(\sum_{\i \in \I^N} \Phi^{s,q}(\i)\right)^{\frac{1}{N}} \\
&\geq& C_N^{\frac{1}{N}}\left(\limsup_{n \to \infty}\left(\sum_{\i \in \I^n} \Phi^{s,q}(\i)\right)^{\frac{1}{n}}-\varepsilon\right)
\end{eqnarray*}
recalling that $C_N$ has the property that $C_N^{\frac{1}{N}} \to 1$ as $N \to \infty$.
Hence
$$\liminf_{k \to \infty} \left(\sum_{\i \in \I^{Nk}} \Phi^{s,q}(\i)\right)^{\frac{1}{Nk}} \geq C_N^{\frac{1}{N}}\left(\limsup_{n \to \infty}\left(\sum_{\i \in \I^n} \Phi^{s,q}(\i)\right)^{\frac{1}{n}}-\varepsilon\right).$$
The proof in the case where $s \geq t_f(q)+t_g(q)$ is complete by observing that
$$\liminf_{k \to \infty} \left(\sum_{\i \in \I^{Nk}} \Phi^{s,q}(\i)\right)^{\frac{1}{Nk}} =\liminf_{N \to \infty} \left(\sum_{\i \in \I^N} \Phi^{s,q}(\i)\right)^{\frac{1}{N}}.$$
If $s<t_f(q)+t_g(q)$ the proof is similar except we choose $N$ such that $\left(\sum_{\i \in \I^N} \Phi^{s,q}(\i)\right)^{\frac{1}{N}}< \liminf_{n \to \infty}\left(\sum_{\i \in \I^n} \Phi^{s,q}(\i)\right)^{\frac{1}{n}}+\varepsilon$, and use the ``almost'' submultiplicativity of $\Phi^{s,q}$ instead.
\end{proof}

By the tempered distortion property (in particular \eqref{td1} and \eqref{td2}) we can replace $\o_1(\i)$ and $\o_2(\i)$ by $\u_1(\i)$ and $\u_2(\i)$ (or any intermediate value) without affecting the definition of the pressure, which is easily seen since for all $q \geq 0$ and $s \in \mathbb{R}$:
\begin{align*}&\, e^{-p_nT(s,q)}\left(\sum_{\i \in \I^n} \u_1(\i)^{t(\i,q)} \u_2(\i)^{s-t(\i,q)}\right)^{\frac{1}{n}} \leq \left(\sum_{\i \in \I^n} \o_1(\i)^{t(\i,q)} \o_2(\i)^{s-t(\i,q)}\right)^{\frac{1}{n}}\\
& \hspace{4cm} \leq e^{p_nT(s,q)}  \left(\sum_{\i \in \I^n} \u_1(\i)^{t(\i,q)} \u_2(\i)^{s-t(\i,q)}\right)^{\frac{1}{n}}
\end{align*}
and $p_n \to 0$. So we know that, provided the pressure for one of these singular value functions exists, all of them exist and coincide, i.e. the minimal roots also coincide with $s_0$. 

It is immediate that for a fixed $q \geq 0$, $P(s,q)$ is decreasing and continuous as a function of $s$.  In the hyperbolic setting the next step is to identify the `unique root' of the pressure, that is, a value $s_0$ such that $P(s_0,q) = 1$.  However, as we saw above, this does not have to exist in the parabolic setting, even in the  setting of $\mathbb{R}$.  

\begin{lma}
For any measure on a parabolic carpet,  $P(2,q) \leq 1$ for all $q \geq 0$.
\end{lma}
\begin{proof}
By A3, for all $n$ the sets 
\[
\{ S_\i ((0,1)^2) : \i \in \I^{n}\}
\]
are pairwise disjoint and contained in $(0,1)^2$.  Moreover, for all $\i \in \I^n$, $t(\i,0) \leq 1$ and therefore
\[
\o_1(\i)^{t(\i,0)} \o_2(\i)^{2-t(\i,0)} \leq \o_1(\i) \o_2(\i)  \leq e^{2np_n} \mathcal{L}^2(S_\i ((0,1)^2))
\]
by Lemma \ref{tdp} where $\mathcal{L}^2$ denotes 2-dimensional Lebesgue measure.  Therefore
\begin{align*}
P(2,q) \leq P(2,0) &= \lim_{n \to \infty} \left(\sum_{\i \in \I^n} \o_1(\i)^{t(\i,0)} \o_2(\i)^{2-t(\i,0)}\right)^{\frac{1}{n}}\\
& \leq \lim_{n \to \infty} e^{2p_n} \left(\sum_{\i \in \I^n}  \mathcal{L}^2(S_\i ((0,1)^2))\right)^{\frac{1}{n}}\\ &\leq  \mathcal{L}^2( (0,1)^2) = 1
\end{align*}
as required.
\end{proof}

The previous lemma allows us to define a function $\beta: [0,\infty) \to \mathbb{R}$ by
$$\beta(q):=\inf\{s>0: P(s,q) \leq 1\}.$$

We will eventually show that $\beta(q)$ gives the $L^q$-spectrum of $\mu$ but in order to do that we need to first show that $\beta$ can also be realised as the critical exponent of a certain zeta function.  This is the content of the next subsection.

\subsection{Coincidence of minimal root of pressure and  critical exponent}

For the rest of this section, fix $q \geq 0$.  Write
$$s_0:=\inf\{s>0: P(s,q) \leq 1\}.$$
We will now give criteria under which the critical exponent $\delta_\Phi$ of the zeta function
$$\zeta_\Phi(s,q):= \sum_{n=1}^\infty \sum_{\i \in \I^n} \Phi(\i)$$
also coincides with $s_0$ for any choice of singular value function $\Phi$ that satisfies Lemma \ref{mult}.

By Lemma \ref{existshyperbolic} (see also \eqref{uniform}),  there exists $\i_0 \in \I^{*}$, such that $S_{\i_0}$ is a uniform contraction (hyperbolic) and $\{S_{\j\i_0}: \j \in \I^*\}$ are a set of uniform contractions (hyperbolics) i.e. there exists $\rho<1$ such that
\begin{equation} \label{rho}
\sup_{\j \in \I^*} \o_1(\j\i_0) \leq \rho.
\end{equation}
 Since $\delta_\I= \delta_{\I^{k}}$ and $s_{\I}=s_{\I^{k}}$ for all $k$, without loss of generality we can assume that $k=1$.

Recall that we defined 
\[
\I_\infty\vcentcolon=\{\j\i_0: \j \in (\I \setminus \{\i_0\})^*\} \cup \{\i_0\}.
\]

\begin{prop}\label{crit}
Let $\mu=\mathbb{P}\circ \Pi^{-1}$ be an almost quasi-Bernoulli measure on a parabolic carpet satisfying 
$$\sum_{\i \in \I_\infty} |f_\i([0,1])|^{\alpha_f}<\infty \; \textnormal{ and} \;  \sum_{\i \in \I_\infty} |g_\i([0,1])|^{\alpha_g}<\infty$$
and let $\Phi$ satisfy \eqref{sing}.  Then $\delta_\Phi=s_0$. In particular, this holds for all parabolic carpets where the maps $f_i,g_i$ are $C^{1+\textup{Lip}}$ or $C^2$.
\end{prop}

Throughout this section the singular value function $\Phi^{s,q}(\i)=\mathbb{P}([\i])^q\alpha_1(\i)^{t(\i,q)}\alpha_2(\i)^{s-t(\i,q)}$ is fixed (where  $\alpha_1(\i)$ denotes the length of the longer side of $S_\i([0,1]^2)$ and $\alpha_2(\i)$ denotes the length of the shorter side of  $S_\i([0,1]^2)$ ) but we will be considering the pressure and zeta functions of several different digit sets. Given a digit set $\J$,  let $s_\J$ denote the minimal root of the pressure $P_\J(s,q)=\lim_{n \to \infty}\left(\sum_{\i \in \J^n} \Phi^{s,q}(\i)\right)^{\frac{1}{n}}$ and $\delta_\J$ denote the critical exponent of the zeta function $\zeta_\Phi(s,q)=\sum_{n=1}^\infty\sum_{\i \in \J^n} \Phi(\i)$. 
Observe that for any finite or countable set $\J \subset \I^*$, we have that $s_\J\leq \delta_\J$.

Denote
\[
\I_N\vcentcolon= \I_\infty \cap \left(\bigcup_{m=1}^N \I^m\right).
\]

Similarly to the proof of Proposition \ref{crit2}, the proof of Proposition \ref{crit} will follow from four lemmas where we will show that:
\[
\mbox{\textbf{(i)}}   \ \ \delta_\I=\delta_{\I_\infty},  \ \ \ \ \ \  
\mbox{\textbf{(ii)}}  \ \ \delta_{\I_\infty}=s_{\I_\infty},  \ \ \ \ \ \ 
\mbox{\textbf{(iii)}} \ \ s_{\I_{\infty}}=\sup_N s_{\I_N}, \ \ \ \ \ \ 
\mbox{\textbf{(iv)}} \ \sup_N s_{\I_N} =s_\I, \ \ \ \ \ \ 
\]
although due to the non-conformality the proofs will be more involved.

Again, in the proofs that follow we'll assume that $\delta_\I<\infty$, but one can easily see that the proofs also imply $s_\I=\infty$ if $\delta_\I=\infty$.

\begin{lma}[Proof of \textbf{(ii)}]\label{ii}
$\delta_{\I_\infty}=s_{\I_\infty}$.
\end{lma}

\begin{proof}
 By the root test,  $P_{\I_\infty}(s,q) \geq 1$, for $s < \delta_{\I_\infty}$, and $P_{\I_\infty}(s,q) \leq 1$, for $s > \delta_{\I_\infty}$. Therefore,  $\delta_{\I_\infty}$ is a root of $P_{\I_\infty}(\cdot, q)$. By \eqref{rho}, $\o_1(\i) \leq \rho^{n}$ for any $\i \in \I_\infty^n$. This implies that $P_{\I_\infty}(\cdot, q)$ has a unique root. To see this, observe
\[
\sum_{\i \in \I_\infty^n} \Phi^{s_{\I_\infty}+\epsilon,q}(\i) \leq \rho^{n\epsilon} \sum_{\i \in \I_\infty^n} \Phi^{s_{\I_\infty},q}(\i).
\]
In particular, $P_{\I_\infty}(s_{\I_\infty}+\epsilon,q) \leq \rho^{\epsilon}\, P_{\I_\infty}(s_{\I_\infty},q)=\rho^\epsilon<1$. Therefore $\delta_{\I_\infty}$ is the unique root of $P_{\I_\infty}$, that is, $s_{\I_\infty}=\delta_{\I_\infty}$ as required.
\end{proof}

\begin{lma}[Proof of  \textbf{(i)}] \label{i}
$\delta_\I=\delta_{\I_\infty}$.
\end{lma}

\begin{proof}
 Since $\I_\infty^* \subseteq \I^*$, we have $\delta_{\I_\infty} \leq \delta_\I$. We claim that there exists a sequence $D_n$ such that $\lim_{n \to \infty} D_n^{\frac{1}{n}}=1$ and $\Phi^{s,q}(\i) \leq D_n \Phi^{s,q}(\i\i_0)$ for all $\i\in \I^n$. Indeed, for all $\i \in \I^n$,
\[
\frac{\mathbb{P}([\i\i_0])^q\alpha_1(\i\i_0)^{t(\i\i_0,q)}\alpha_2(\i\i_0)^{s-t(\i\i_0,q)}}{\mathbb{P}([\i ])^q\alpha_1(\i)^{t(\i,q)}\alpha_2(\i)^{s-t(\i,q)}} \gtrsim_{\i_0, q, s} c_n^{-q}
 e^{-T(s,q)np_n}  \left(\frac{\alpha_1(\i)}{\alpha_2(\i)}\right)^{t(\i\i_0,q)-t(\i,q)}
\]
where $T(s,q) \geq 0$ is the constant from \eqref{sillyconstant}.  For this we need to use  \eqref{wqbp}, Lemma \ref{tdp}, super- and submultiplicativity of $\o_1$ and $\u_2$ respectively and the fact that $\alpha_1(\j\k)\geq \u_1(\j)\u_2(\k)$ and $\alpha_2(\j\k)\leq \o_1(\j)\o_2(\k)$  for all  words $\j,\k \in \I^*$.  Precisely which estimates we apply depend on the signs of $s-t( \cdot, q)$ and $t( \cdot, q)$.  

If $t(\i\i_0,q)-t(\i,q) \geq 0$ then we are done. Otherwise, observe that 
$$\u_1(\i)\u_2(\i_0) \leq\o_2(\i)\o_1(\i_0).$$
Hence
\begin{align*}
 \left(\frac{\alpha_2(\i)}{\alpha_1(\i)}\right)^{t(\i,q)-t(\i\i_0,q)} \geq  \left(\frac{e^{-np_n}\o_2(\i)}{e^{np_n}\u_1(\i)}\right)^{t(\i,q)-t(\i\i_0,q)} &\geq  \left(e^{-2np_n}\frac{\u_2(\i_0)}{\o_1(\i_0)}\right)^{t(\i,q)-t(\i\i_0,q)}\\ &= \left(e^{-2np_n}\frac{\u_2(\i_0)}{\o_1(\i_0)}\right)^{|t_f(q)-t_g(q)|}
\end{align*}
which proves the claim.

Therefore 
\begin{align}
\zeta_\I(s,q) &= \sum_{n=1}^{\infty} \sum_{\substack{\i \in \I^n \\ i_n=\i_0}} \Phi^{s,q}(\i)+\sum_{n=1}^{\infty} \sum_{\substack{\i \in \I^n \\ i_n\neq \i_0}} \Phi^{s,q}(\i) \nonumber\\
&\leq \zeta_{\I_\infty}(s,q) +    \sum_{n=1}^{\infty} D_n\sum_{\substack{\i \in \I^n \\ i_n\neq \i_0}} \Phi^{s,q}(\i\i_0). \label{zeta ineq}
\end{align}
If $s> \delta_{\I_\infty}$ then $ \zeta_{\I_\infty}(s,q)<\infty$. Moreover, as seen in the proof of Lemma \ref{ii}, 
$$ \lim_{n \to \infty}\left(\sum_{\substack{\i \in \I^n \\ i_n\neq \i_0}} \Phi^{s,q}(\i\i_0)\right)^{\frac{1}{n}} \leq P_{\I_\infty}(s,q)<1$$
and since $\lim_{n \to\infty} D_n^{\frac{1}{n}}=1$ we have that the right hand side of \eqref{zeta ineq} is summable for $s>\delta_{\I_\infty}$.  In particular $\delta_\I \leq \delta_{\I_\infty}$.
\end{proof}

\begin{lma}[Proof of  \textbf{(iii)}] \label{iii}
$s_{\I_{\infty}}=\sup_N s_{\I_N}.$
\end{lma}

\begin{proof}
Since $\I_N \subseteq \I_\infty$, clearly $\sup_N s_{\I_N} \leq s_{\I_\infty}$. Therefore it is sufficient to show that for all $\epsilon>0$ there exists $N \in \N$ such that $s_{\I_N}>s_{\I_\infty}-\epsilon$. Recall that the family $\{S_\i\}_{\i \in \I_\infty}$ satisfies the bounded distortion property \eqref{bdp} (in both coordinates) and $\mathbb{P}$ is quasi-Bernoulli on $\I_\infty^*$ in the sense of \eqref{qbp}, both of which will be crucial for this proof. 

 Fix $\epsilon>0$ and write $s=s_{\I_\infty}-\epsilon$. Let $1<\lambda<P_{\I_\infty}(s_{\I_\infty}-\epsilon, q)$ and choose $n$ sufficiently large that 
$$\left(\sum_{\i \in \I_\infty^n} \Phi^{s,q}(\i)\right)^{\frac{1}{n}}>\frac{\lambda+P_{\I_\infty}(s_{\I_\infty}-\epsilon, q)}{2}$$ and $$\left(\frac{1}{3c^qC^{2T(s,q)+2|t_f(q)-t_g(q)|}}\right)^{\frac{1}{n}}>\frac{1}{\lambda}$$
Since $\lim_{N \to \infty} \left(\sum_{\i \in \I_N^n}\Phi^{s,q}(\i)\right)^{\frac{1}{n}}=\left(\sum_{\i \in \I_\infty^n}\Phi^{s,q}(\i)\right)^{\frac{1}{n}}>\frac{\lambda+P_{\I_\infty}(s_{\I_\infty}-\epsilon,q)}{2}>\lambda,$ we can choose $N$ sufficiently large that
$$ \left(\sum_{\i \in \I_N^n}\Phi^{s,q}(\i)\right)^{\frac{1}{n}}>\lambda.$$
For each $\i \in \I_N^n$ either:
\begin{enumerate}[(a)]
\item $\max_x|f'_\i(x)|<\min_x|g'_\i(x)|$, or
\item $\max_x|g'_\i(x)|<\min_x|f'_\i(x)|$ or
\item $\alpha_2(\i) \leq \alpha_1(\i) \leq C^2\alpha_2(\i)$. (i.e. in this category we either have $\min_x|f'_\i(x)|<\max_x|g'_\i(x)|<\max_x|f'_\i(x)|$ or $\min_x|g'_\i(x)|<\max_x|f'_\i(x)|<\max_x|g'_\i(x)|$, then apply bounded distortion property \eqref{bdp} potentially twice).
\end{enumerate}
  Hence we can choose $\Gamma \subset \I_N^n$ such that
$$\sum_{\i\in \Gamma} \Phi^{s,q}(\i) \geq \frac{1}{3}\sum_{\i \in \I_N^n}\Phi^{s,q}(\i)$$
and all $\i \in \Gamma$ belong to the same category (a), (b) or (c). We claim that if $\i_1, \ldots, \i_k \in \Gamma$, then $\Phi^{s,q}(\i_1 \ldots \i_k) \geq C_*^{-k}\Phi^{s,q}(\i_1) \cdots \Phi^{s,q}(\i_k)$ where $C_*:=c^qC^{2T(s,q)+2|t_f(q)-t_g(q)|}$. To see this, first assume all $\i \in \Gamma$ belong to either category (a) or (b). Then since $\o_1$ and $\o_2$ are supermultiplicative  and $\u_1$ and $\u_2$ are submultiplicative on $\Gamma$ we can bound below
\begin{eqnarray*}
&\,& \hspace{-2cm} \frac{\mathbb{P}([\i_1 \ldots \i_k])^q\alpha_1(\i_1 \ldots \i_k)^{t(\i_1 \ldots \i_k,q)}\alpha_2(\i_1 \ldots \i_k)^{s-t(\i_1 \ldots \i_k,q)}}{\mathbb{P}([\i_1  ])^q\alpha_1(\i_1)^{t(\i_1,q)}\alpha_2(\i_1)^{s-t(\i_1,q)} \cdots \mathbb{P}([ \i_k])^q\alpha_1(\i_k)^{t(\i_k,q)}\alpha_2(\i_k)^{s-t(\i_k,q)}} \\
&\geq& c^{-kq}C^{-kT(s,q)} \prod_{k=1}^k\frac{\alpha_1(\i_j)^{t(\i_1 \ldots \i_k,q)}\alpha_2(\i_j)^{s-t(\i_1 \ldots \i_k,q)}}{\alpha_1(\i_j)^{t(\i_j,q)}\alpha_2(\i_j)^{s-t(\i_j,q)}}.
\end{eqnarray*}
We are done since $t(\i_1 \ldots \i_k,q)=t(\i_j,q)$ for all $1 \leq j \leq k$.

On the other hand, if all $\i \in \Gamma$ belong to category (c) then 
\begin{align*}
&\, \hspace{-2cm}\frac{\mathbb{P}([\i_1 \ldots \i_k])^q\alpha_1(\i_1 \ldots \i_k)^{t(\i_1 \ldots \i_k,q)}\alpha_2(\i_1 \ldots \i_k)^{s-t(\i_1 \ldots \i_k,q)}}{\mathbb{P}([\i_1 ])^q\alpha_1(\i_1)^{t(\i_1,q)}\alpha_2(\i_1)^{s-t(\i_1,q)} \cdots \mathbb{P}([  \i_k])^q \alpha_1(\i_k)^{t(\i_k,q)}\alpha_2(\i_k)^{s-t(\i_k,q)}}\\
& \geq  c^{-kq} \prod_{j=1}^k\frac{\u_2(\i_j)^{t(\i_1 \ldots \i_k,q)}\u_2(\i_j)^{s-t(\i_1 \ldots \i_k,q)}}{\alpha_1(\i_j)^{t(\i_j,q)}\alpha_2(\i_j)^{s-t(\i_j,q)}}\\
&\geq c^{-kq}C^{-kT(s,q)} \prod_{j=1}^k\frac{\alpha_2(\i_j)^{t(\i_1 \ldots \i_k,q)}\alpha_2(\i_j)^{s-t(\i_1 \ldots \i_k,q)}}{\alpha_1(\i_j)^{t(\i_j,q)}\alpha_2(\i_j)^{s-t(\i_j,q)}}\\
&\geq c^{-kq}C^{-k(T(s,q)+2\max\{|t_f(q)|,|t_g(q)|\})}\prod_{j=1}^k\frac{\alpha_1(\i_j)^{t(\i_1 \ldots \i_k,q)}\alpha_2(\i_j)^{s-t(\i_1 \ldots \i_k,q)}}{\alpha_1(\i_j)^{t(\i_j,q)}\alpha_2(\i_j)^{s-t(\i_j,q)}} \\
&\geq c^{-kq}C^{-2kT(s,q)}\prod_{j=1}^k\left(\frac{\alpha_1(\i_j)}{\alpha_2(\i_j)}\right)^{t(\i_1 \ldots \i_k,q)-t(\i_j,q)}.
\end{align*}
For each $j$ such that $t(\i_1 \ldots \i_k,q)-t(\i_j,q) \geq 0$ the $j$th term in the product above is $\geq 1$.  On the other hand,  if $t(\i_1 \ldots \i_k,q)-t(\i_j,q) < 0$ then, since $\alpha_2(\i) \geq C^{-2}\alpha_1(\i)$, we have
$$\left(\frac{\alpha_2(\i_j)}{\alpha_1(\i_j)}\right)^{t(\i_j,q)-t(\i_1 \ldots \i_k,q)} \geq C^{-2|t_f(q)-t_g(q)|}$$
which completes the proof of the claim.

 In particular
\begin{eqnarray*}
P_{\I_N}(s,q)=\lim_{k \to \infty}\left(\sum_{\i \in \I_N^{nk}} \Phi^{s,q}(\i)\right)^{\frac{1}{nk}}&\geq& \lim_{k \to \infty}\left(\sum_{\i \in \Gamma^{k}} \Phi^{s,q}(\i)\right)^{\frac{1}{nk}} \\
&\geq&  \left(\sum_{\i \in \Gamma}C_*^{-1} \Phi^{s,q}(\i)\right)^{\frac{1}{n}}\\
&\geq&\left(\frac{1}{3C_*}\right)^{\frac{1}{n}} \left(\sum_{\i \in \I_N^n}\Phi^{s,q}(\i)\right)^{\frac{1}{n}}>\frac{1}{\lambda} \cdot \lambda=1.
\end{eqnarray*}
In particular $s_{\I_N}>s=s_{\I_\infty}-\epsilon$, completing the proof of \textbf{(iii)}. \end{proof}

\begin{lma}[Proof of \textbf{(iv)}]
$\sup_N s_{\I_N}=s_\I$.
\end{lma}

\begin{proof}
Note that since $s_\I \leq \delta_\I= \sup_N s_{\I_N}$ by the conditions \textbf{(i-iii)} above, it is sufficient to prove that $\sup_N s_{\I_N} \leq s_\I$. We first prove this under the assumption that $s_\I \leq t_f(q)+t_g(q)$. Let $N \in \N$. Then by almost submultiplicativity of the singular value function,
\begin{align*}
\left(\sum_{\i \in \I_N^{k\ell}} \Phi^{s_\I,q}(\i)\right)^{\frac{1}{Nk\ell}}& \leq \left(\sum_{\i \in \I^{Nk\ell}} \Phi^{s_\I,q}(\i)\right)^{\frac{1}{Nk\ell}}\\
&\leq \left(\left(\sum_{\i \in \I^{N\ell}} \Phi^{s_\I,q}(\i)\right)^kC_{N\ell}^k\right)^{\frac{1}{Nk\ell}}\\
&= C_{N\ell}^{\frac{1}{N\ell}}\left(\sum_{\i \in \I^{N\ell}} \Phi^{s_\I,q}(\i)\right)^{\frac{1}{N\ell}}.
\end{align*}
Hence
$$P_{\I_N}(s,q)^{\frac{1}{N}}=\lim_{\ell \to \infty}\left(\sum_{\i \in \I_N^{k\ell}} \Phi^{s_\I,q}(\i)\right)^{\frac{1}{Nk\ell}}\leq \lim_{\ell \to \infty}\left(\sum_{\i \in \I^{N\ell}} \Phi^{s_\I,q}(\i)\right)^{\frac{1}{N\ell}}=P_\I(s_\I,q)=1.$$
In particular $P_{\I_n}(s_\I) \leq 1$ hence $s_{\I_N} \leq s$, completing the proof in the case where $s_\I \leq t_f(q)+t_g(q)$.

Next, we prove that  $\sup_N s_{\I_N} \leq s_\I$ in the case where $s_\I > t_f(q)+t_g(q)$, so the singular value function is only almost supermultiplicative (this is the harder case). For a contradiction we assume that $s_{\I} < s_{\I_N}$ for some $N \in \N$. Then we can choose $C>1$ and $n_0 \in \N$ such that 
\begin{equation}\label{C}
\left(\sum_{\i \in \I_N^n} \Phi^{s_\I,q}(\i)\right)^{\frac{1}{n}} \geq C,
\end{equation}
for all $n \geq n_0$. Choose $\epsilon>0$ sufficiently small such that $C^{\frac{1}{N}}(1-\epsilon) >1$. Note that there exists $M \in \N$ such that for all $m \geq M$, 
\begin{equation}\label{eps}
\left(\sum_{\i \in \I^m} \Phi^{s_\I,q}(\i) \right)^{\frac{1}{m}} \geq 1-\epsilon. 
\end{equation}
Given $\i \in \I_N^k$, write $|\i|=l$ if $\i\in\I^l$ and note that for any $\i \in \I_N^k$, $k\leq |\i| \leq Nk$. Now, observe that for any $k \in \N$,
\begin{align}\label{rewrite}
\sum_{\i \in \I^{Nk}}\Phi^{s_\I,q}(\i) \geq &\sum_{\substack{\i \in \I_N^k\\ |\i|=Nk}} \Phi^{s_\I,q}(\i)+ \sum_{\substack{\i \in \I_N^k\\ |\i|=Nk-1}}\Phi^{s_\I,q}(\i_0\i) \,+ \nonumber\\
&+  \sum_{\substack{\i \in \I_N^k\\ k \leq |\i|\leq Nk-2}} \sum_{\j \in \I^{Nk-1-|\i|}} \Phi^{s_\I,q}(\j\i_0\i). 
\end{align}

Put $c'= \min_{1 \leq m \leq M-1} \sum_{\i \in \I^m} \Phi^{s_\I,q}(\i),$ and observe that by \eqref{eps} we have that for any $l \leq Nk-2$,
\[
\sum_{\i \in \I^{Nk-1-l}}  \Phi^{s_\I,q}(\i) \geq \min \{c', (1-\epsilon)^{Nk-1-l}\} \geq \min\{c', (1-\epsilon)^{Nk}\}=(1-\epsilon)^{Nk},\]
whenever $k \geq k_0$ for some sufficiently large $k_0$. Then, from (\ref{rewrite}) for any  $k \geq k_0$ 
\begin{align*}
\sum_{\i \in \I^{Nk}}  \Phi^{s_\I,q}(\i) &\geq \sum_{\substack{\i \in \I_N^k\\ |\i|=Nk}}  \Phi^{s_\I,q}(\i)+ C_{Nk-1}^{-1}C_1^{-1} \Phi^{s_\I,q}(\i_0)\left(\sum_{\substack{\i \in \I_N^k\\ |\i|=Nk-1}}  \Phi^{s_\I,q}(\i)\right) +\\
& + \Phi^{s_\I,q}(\i_0) \sum_{\substack{\i \in \I_N^k\\ k \leq |\i|\leq Nk-2}}C_{|\i|}^{-1}C_1^{-1}C_{Nk-1-|\i|}^{-1} \Phi^{s_\I,q}(\i)\left(\sum_{\j \in \I^{Nk-1-|\i|}} \Phi^{s_\I,q}(\j)\right)\\
&\geq  \Phi^{s_\I,q}(\i_0)C_1^{-1}C_{Nk}^{-2}(1-\epsilon)^{Nk}\sum_{\i \in \I_N^k}  \Phi^{s_\I,q}(\i).
\end{align*}
 In particular by (\ref{C}), for any $k \geq \min\{n_0, k_0\}$,
\begin{align*}
\left(\sum_{\i \in \I^{Nk}}  \Phi^{s_\I,q}(\i) \right)^{\frac{1}{Nk}} &\geq (C_1C_{Nk}^2 )^{-\frac{1}{Nk}} (\Phi^{s_\I,q}(\i_0))^{\frac{1}{Nk}}(1-\epsilon)\left(\sum_{\i \in \I_N^k}  \Phi^{s_\I,q}(\i)\right)^{\frac{1}{Nk}}\\
&\geq C^{\frac{1}{N}} (C_1C_{Nk}^2 )^{-\frac{1}{Nk}} (1-\epsilon)\left(\Phi^{s_\I,q}(\i_0) \right)^{\frac{1}{Nk}}.
\end{align*}
In particular, $P_{\I}(s_{\I},q)=\lim_{k \to \infty}\left(\sum_{\i \in \I^{Nk}}  \Phi^{s_\I,q}(\i) \right)^{\frac{1}{Nk}}\geq  C^{\frac{1}{N}} (1-\epsilon)  >1$, giving a contradiction. Thus the proof of \textbf{(iv)} is complete. \end{proof}

\section{$L^q$-spectrum of almost quasi-Bernoulli measures on parabolic carpets} \label{proof-carpet}

The following theorem constitutes our main result.

\begin{thm} \label{carpetthm}
If  $\mu$ is an almost quasi-Bernoulli measure on a parabolic carpet $F$ for which
$$\sum_{\i \in \I_\infty} |f_\i([0,1])|^{\alpha_f}<\infty \; \textnormal{ and} \;  \sum_{\i \in \I_\infty} |g_\i([0,1])|^{\alpha_g}<\infty$$
    then 
$$\tau_\mu(q) = \beta(q)$$
for all $q \geq 0$.  Moreover, $\bd F = \pd F = \beta(0)$.   In particular, these results  hold for all parabolic carpets where the maps $f_i,g_i$ are $C^{1+\textup{Lip}}$ or $C^2$.
\end{thm}

We remark that the box dimension formula follows immediately from the $L^q$-spectrum result and the fact that this also coincides with the packing dimension follows since every open ball intersecting the carpet contains a bi-Lipschitz copy of the carpet allowing us to apply, e.g., \cite[Corollary 3.9]{falconer}.

\subsection{Proof of upper bound for $\tau_\mu$ in Theorem \ref{carpetthm}}
Let $q \geq 0$ and  $\epsilon>0$.  Given $0<\delta<1$ define the $\delta$-stopping $\mathcal{S}_\delta:=\{\i \in \I^*: \ell(\i)\leq \delta<\ell(\i^-)\}$, where $\ell(\i)$ denotes the length of the shorter side of $S_\i([0,1]^2)$ and where $\i^-$ is obtained from $\i$ by removing the final symbol in $\I$.  Note that, by A2,  $\ell(\i)\approx  \delta$ for $\i \in \mathcal{S}_\delta$.

Let $\i \in \mathcal{S}_\delta$.  If $q \in [0,1]$, then   
$$ D^q_\delta(\mathbb{P}([\i]) S_\i(\mu)) \lesssim D^q_{\frac{\ell(\i)}{\o_1(\i)}}(\mathbb{P}([\i])\pi_\i \mu)  \lesssim \mathbb{P}([\i])^q\left(\frac{\o_1(\i)}{\ell(\i)}\right)^{t(\i, q)+\epsilon/2}\lesssim \mathbb{P}([\i])^q \left(\frac{\o_1(\i)}{\delta}\right)^{t(\i, q)+\epsilon/2}.$$
Here the implicit constants may depend on $q$ and $\epsilon$ but not on $\i$ or $\delta$.  On the other hand, if $q >1$, then  similarly
$$ D^q_\delta(\mathbb{P}([\i]) S_\i(\mu)) \lesssim D^q_{\frac{\ell(\i)}{\u_1(\i)}}(\mathbb{P}([\i])\pi_\i \mu)  \lesssim \mathbb{P}([\i])^q\left(\frac{\u_1(\i)}{\ell(\i)}\right)^{t(\i, q)+\epsilon/2}\lesssim \mathbb{P}([\i])^q \left(\frac{\u_1(\i)}{\delta}\right)^{t(\i, q)+\epsilon/2}.$$
For the rest of the proof we will write $\alpha_1$ to denote $\o_1$ if $q \in [0,1]$ and $\u_1$ if $q >1$.  Next we relate the moments $D^q_\delta(\mu)$ to sums of moments of component measures.  It is a simple but  crucial observation  that a given $\delta$-square $Q \in \mathcal{Q}_\delta$   intersects at most $\lesssim 1$ of the sets $S_\i([0,1]^2)$ for $\i \in \mathcal{S}_\delta$  This uses the definition of $\mathcal{S}_\delta$ and A3.  This is why we get approximate \emph{equality} below when we apply Jensen's inequality.  Then
\begin{eqnarray*}
D^q_\delta(\mu) &=& \sum_{Q \in \mathcal{Q}_\delta} \mu(Q)^q \\
&=& \sum_{Q \in \mathcal{Q}_\delta} \left(\sum_{\i \in \mathcal{S}_\delta} \mathbb{P}([\i]) S_\i(\mu)(Q) \right)^q \\
	&\approx& \sum_{Q \in \mathcal{Q}_\delta} \sum_{\i \in  \mathcal{S}_\delta} \mathbb{P}([\i])^q S_\i(\mu)(Q)^q \qquad \text{(by Jensen's inequality and separation A3)} \\
&=& \sum_{\i \in  \mathcal{S}_\delta}\sum_{Q \in \mathcal{Q}_\delta}  \mathbb{P}([\i])^q S_\i(\mu)(Q)^q \\
&=& \sum_{\i \in  \mathcal{S}_\delta}D^q_\delta(\mathbb{P}([\i])S_\i(\mu)).
\end{eqnarray*}
Combining we have
\begin{eqnarray*}
\delta^{\beta(q)+\epsilon}D^q_\delta(\mu) &\lesssim& \delta^{\beta(q)+\epsilon}\sum_{\i \in \mathcal{S}_\delta}\mathbb{P}([\i])^q  \left(\frac{\alpha_1(\i)}{\delta}\right)^{t(\i, q)+\epsilon/2} \\
&\lesssim& \sum_{\i \in \mathcal{S}_\delta} \mathbb{P}([\i])^q  \alpha_1(\i)^{t(\i,q)} \ell(\i)^{\beta (q)+\epsilon/2-t(\i,q)}\\
&\leq&  \sum_{\i \in \mathcal{S}_\delta}\mathbb{P}([\i])^q  \alpha_1(\i)^{t(\i,q)} \o_2(\i)^{\beta (q)+\epsilon/2-t(\i,q)} \\
&\leq& \sum_{n \in \N}\sum_{\i \in \I^n} \mathbb{P}([\i])^q \alpha_1(\i)^{t(\i,q)} \o_2(\i)^{\beta (q)+\epsilon/2-t(\i,q)}< \infty
\end{eqnarray*}
 by Proposition \ref{crit}. This proves  $\tau_\mu(q) \leq \beta(q)$.

\subsection{Proof of lower bound in  Theorem \ref{carpetthm}}
Let $q \geq 0$ and  $\epsilon>0$.   Choose $N$ sufficiently large such that $\beta(q)-\epsilon/2<s_N<\beta(q)$ where $s_N = s_{\mathcal{I}_N}$ is the root of the pressure associated with the hyperbolic subsystem $\mathcal{I}_N$. We can choose such an $N$ by Proposition \ref{crit} and Lemma \ref{iii}. 

Fix $0<\delta< 1$. Define $\mathcal{S}^N_\delta:=\{\i \in (\I_N)^*: \ell(\i)\leq \delta<\ell(\i^-)\}$, noting   that $\ell(\i)\approx_N \delta$ for $\i \in \mathcal{S}^N_\delta$. This time $\i^-$ is obtained from $\i$ by removing the final symbol in $\I_N$.

\begin{lma}\label{lowerlemma}
We have
\[
 \sum_{\i \in \mathcal{S}^N_\delta}\mathbb{P}([\i])^q  \o_1(\i)^{t(\i,q)} \o_2(\i)^{\beta(q)-\epsilon/2-t(\i,q)} \gtrsim 1.
\]
\end{lma}
\begin{proof}
Write $s=\beta(q)-\epsilon/2$.  We split the proof into two cases depending on   $s$. Throughout we use $\Phi^{s,q}(\i) = \mathbb{P}([\i])^q  \o_1(\i)^{t(\i,q)} \o_2(\i)^{s-t(\i,q)}$.
\\

Case 1: $0 \leq s \leq t_f(q)+t_g(q)$. \\

In this case we have  almost submultiplicativity, that is, there exists a constant $c'=c'(q)$ such that for all $\i, \j \in \mathcal{I}_N^*$
\begin{equation} \label{submm}
\Phi^{s,q}(\i\j) \leq  c' \Phi^{s,q}(\i) \Phi^{s,q}(\j) .
\end{equation}
However, we can upgrade this to genuine submultiplicativity by replacing $\Phi^{s,q}$ with $c' \Phi^{s,q}$, which we do for the remainder of case 1.  Suppose 
\[
 \sum_{\i \in \mathcal{S}^N_\delta}\Phi^{s,q}(\i) \leq 1.
\]
Therefore, using \eqref{submm}, for all $\i \in \mathcal{I}_N^*$
\begin{equation} \label{iterate}
 \sum_{\j \in \mathcal{S}^N_\delta}\Phi^{s,q}(\i\j) \leq   \sum_{\j \in \mathcal{S}^N_\delta}\Phi^{s,q}(\i) \Phi^{s,q}(\j) \leq   \Phi^{s,q}(\i) . 
\end{equation}
Let $k \geq 1$ be a very large integer. We approximate words in $\I_N^k$ by compositions of words in $\mathcal{S}^N_\delta$ by defining
\[
\mathcal{S}^{N,k}_\delta = \{ \i_1 \cdots \i_m : \forall l=1, \dots, m,  \,  \i_l \in \mathcal{S}^N_\delta, \  \exists \i_{m+1} \in \mathcal{S}^N_\delta \ \text{s.t.} \ |\i_1 \cdots \i_m| \leq k < |\i_1 \cdots \i_m\i_{m+1}| \}
\]
where here $|\cdot|$ denotes length as a word in the alphabet $\I_N$ (rather than length as a word in the alphabet $\I$). Then, by iteratively applying \eqref{iterate},
\[
 \sum_{\i \in \mathcal{S}^{N,k}_\delta}\Phi^{s,q}(\i) \leq 1.
\]
For all $\i \in (\mathcal{I}_N)^k$ we can write $\i = \i_1 \i_2$ for some $\i_1 \in \mathcal{S}^{N,k}_\delta$ and $\i_2 \in \mathcal{I}_N^*$ with $|\i_2| \lesssim_\delta 1$.  Therefore
\[
 \sum_{\i \in  (\mathcal{I}_N)^k}\Phi^{s,q}(\i) \lesssim_\delta   \sum_{\i \in  \mathcal{S}^{N,k}_\delta}\Phi^{s,q}(\i)  \leq 1
\]
which proves (letting $k \to \infty$ while keeping $\delta$ fixed) that $s \geq s_N$, a contradiction. \\

Case 2: $s> t_f(q)+t_g(q)$. \\

In this case we have (almost) supermultiplicativity, that is, for all $\i, \j \in \mathcal{I}_N^*$
\begin{equation} \label{supmm}
\Phi^{s,q}(\i\j) \geq  C' \Phi^{s,q}(\i) \Phi^{s,q}(\j)
\end{equation}
for some constant $C'=C'(q)$.  However, we can upgrade this to genuine supermultiplicativity (that is, with $C'=1$) simply by replacing $\Phi^{s,q}$ with $ C'\Phi^{s,q}$.  Therefore we may assume $C'=1$ in \eqref{supmm}.  

Since $s<s_N$ we may fix $k \in \mathbb{N}$ such that 
\[
\sum_{\i \in  (\mathcal{I}_N)^k} \Phi^{s,q}(\i)  \geq 1.
\]
Then, for all $\i \in \mathcal{I}^*$,
\begin{equation} \label{iterate2}
 \sum_{\j \in (\mathcal{I}_N)^k}\Phi^{s,q}(\i\j) \geq   \sum_{\j \in (\mathcal{I}_N)^k}\Phi^{s,q}(\i) \Phi^{s,q}(\j) \geq   \Phi^{s,q}(\i) . 
\end{equation}
Let $\delta>0$ be very small.  This time we approximate words in $\mathcal{S}_\delta^N$  by compositions of words of length $k$ by defining
\begin{align*}
(\mathcal{I}_N)^k_\delta = \{ \i_1 \cdots \i_m &: \forall l=1, \dots, m,  \,  \i_l \in (\mathcal{I}_N)^k, \\
& \,   \exists \i_{m+1} \in (\mathcal{I}_N)^k \ \text{s.t.} \ \o_2(\i_1 \cdots \i_m) \geq \delta > \o_2(\i_1 \cdots \i_m\i_{m+1}) \}.
\end{align*}
Then, by iteratively applying \eqref{iterate2},
\[
 \sum_{\i \in (\mathcal{I}_N)^k_\delta}\Phi^{s,q}(\i) \geq 1.
\]
For all $\i \in \mathcal{S}^N_\delta$ we can write $\i = \i_1 \i_2$ for some $\i_1 \in (\mathcal{I}_N)^k_\delta$ and $\i_2 \in \mathcal{I}_N^*$ with $|\i_2| \lesssim_k 1$.  Therefore
\[
 \sum_{\i \in \mathcal{S}^N_\delta}\Phi^{s,q}(\i) \gtrsim_k   \sum_{\i \in (\mathcal{I}_N)^k_\delta}\Phi^{s,q}(\i) \geq 1,
\]
completing the proof.
\end{proof}

We can now prove  the lower bound in Theorem \ref{carpetthm}.   Let $\i \in \mathcal{S}_\delta$.  If $q \in [0,1]$, then 
$$ D^q_\delta(\mathbb{P}([\i])S_\i(\mu)) \gtrsim D^q_{\frac{\ell(\i)}{\u_1(\i)}}(\mathbb{P}([\i])\pi_\i \mu)  \gtrsim \mathbb{P}([\i])^q \left(\frac{\u_1(\i)}{\ell(\i)}\right)^{t(\i, q)-\epsilon/2}\gtrsim \mathbb{P}([\i])^q \left(\frac{\u_1(\i)}{\delta}\right)^{t(\i, q)-\epsilon/2}.$$
Again the implicit constants may depend on $q$ and $\epsilon$ but not on $\i$ or $\delta$. On the other hand, if $q >1$, then  similarly
$$ D^q_\delta(\mathbb{P}([\i])S_\i(\mu)) \gtrsim D^q_{\frac{\ell(\i)}{\o_1(\i)}}(\mathbb{P}([\i])\pi_\i \mu)  \gtrsim \mathbb{P}([\i])^q \left(\frac{\o_1(\i)}{\ell(\i)}\right)^{t(\i, q)-\epsilon/2}\gtrsim \mathbb{P}([\i])^q \left(\frac{\o_1(\i)}{\delta}\right)^{t(\i, q)-\epsilon/2}.$$
For the rest of the proof of the lower bound, we   write $\alpha_1$ to denote $\u_1$ if $q \in [0,1]$ and $\o_1$ if $q >1$.  Note that this is the reverse of how the notation $\alpha_1$ was used when proving the upper bound.   Using an analogous argument  from the proof of the upper bound (replacing $\mathcal{S}_\delta$ with $\mathcal{S}_\delta^N$) we get  
\begin{eqnarray*}
D^q_\delta(\mu)  \approx_N   \sum_{\i \in  \mathcal{S}^N_\delta}D^q_\delta(\mathbb{P}([\i])S_\i(\mu)).
\end{eqnarray*}
Combining, we have
\begin{eqnarray*}
\delta^{\beta(q)-\epsilon}D^q_\delta(\mu) &\gtrsim& \delta^{\beta(q)-\epsilon}\sum_{\i \in \mathcal{S}^N_\delta} \mathbb{P}([\i])^q\left(\frac{\alpha_1(\i)}{\delta}\right)^{t(\i, q)-\epsilon/2} \\
&\gtrsim& \sum_{\i \in \mathcal{S}^N_\delta} \mathbb{P}([\i])^q \alpha_1(\i)^{t(\i,q)} \ell(\i)^{\beta (q)-\epsilon/2-t(\i,q)}\\
&\gtrsim& \sum_{\i \in \mathcal{S}^N_\delta}\mathbb{P}([\i])^q \alpha_1(\i)^{t(\i,q)} \o_2(\i)^{\beta(q)-\epsilon/2-t(\i,q)} \gtrsim 1.
\end{eqnarray*}
 by Lemma  \ref{lowerlemma}. This proves  $\tau_\mu(q) \geq \beta(q)$.

\section{Further directions and questions} \label{further}

\subsection{Weakening our assumptions}

Assumptions A1 and A2 are  mild and completely natural in this setting.  The main purpose of A3 is to ensure that the projections are parabolic Cantor sets satisfying A3' (the OSC).  Without this, we are unable to compute the $L^q$-spectrum of the projections, or show that they exist.  Peres and Solomyak \cite{peressolomyak} proved that the $L^q$-spectrum exists for Bernoulli measures on hyperbolic systems without assuming the OSC.  It would be interesting to generalise their result to almost quasi-Bernoulli measures on parabolic Cantor sets with overlaps, but we have not pursued this.   Further to A1, A2 and A3, our main technical assumption in this paper is that the  induced subsystem has good distortion properties, stated succinctly in \eqref{sum} and the need for the quasi-Bernoulli property on the induced system \eqref{qbp}.  It would be interesting to try to remove these assumptions.  The main use of \eqref{sum} is to establish that the minimal root of the pressure coincides with the critical exponent of the zeta function, Proposition \ref{crit}.  Without assuming \eqref{sum} our methods provide non-trivial upper and lower bounds for the $L^q$-spectrum involving the  critical exponent and the minimal root of the pressure, respectively, although we have not stated these formally.

\subsection{Hausdorff dimension and measure}

It would be interesting to study the Hausdorff dimension and measure for parabolic carpets.  We would expect a variational principle, that is, the Hausdorff dimension of the parabolic carpet is the supremum of the Hausdorff dimension of invariant measures supported by it. Further, we would expect the Hausdorff dimension of an invariant measure to be given by a Ledrappier-Young formula.  The question of Hausdorff measure is then of particular interest to us and may be subtle.  Recall, Peres \cite{pereshausdorff} proved that the Hausdorff measure in the Hausdorff dimension of a Bedford-McMullen carpet with non-uniform fibres is infinite.  This might suggest that the same is true for parabolic carpets.  However, for a large family of parabolic Cantor sets, Urba\'nski \cite{urbanski} proved  that the Hausdorff measure in the Hausdorff dimension is zero.  Since the projections of parabolic carpets onto the coordinate axes are parabolic Cantor sets, this might suggest that the Hausdorff measure of the carpet is also zero (or at least not infinite).  To be concrete, let $F$ be a parabolic carpet with parabolicity only in the horizontal direction  and is (strictly) dominated in the sense that
\[
\max_{ i \in \mathcal{I}} \sup_{x \in [0,1]} |g_i'(x)| < \min_{ i \in \mathcal{I}} \inf_{x \in [0,1]} |f_i'(x)|.
\]
This means only the projection onto the horizontal axes is relevant and this model most accurately aligns with the Bedford-McMullen model.  Finally, assume that the projection of $F$ onto the horizontal axes is a parabolic Cantor set with Hausdorff measure zero in the Hausdorff dimension.  Then we ask, is the Hausdorff measure in the Hausdorff dimension of $F$ zero, positive and finite, or infinite?

\subsection{Existence and uniqueness of (weak) Gibbs-type measures}

It would be interesting to study the existence and uniqueness of weak Gibbs-type measures for potentials $\Phi: \Sigma^* \to \R$ such as the ``singular value potentials'' studied in this paper. The analogue of such measures in the conformal parabolic IFS literature are conformal measures, which are known to exist provided that the derivatives of compositions of parabolic maps decay sufficiently fast \cite{parifs}, and therefore we would expect the existence of weak Gibbs-type measures to be related to the rates of decay of $\Phi(i^n)$ for parabolic indices $i$. However, even in the simplest non-conformal hyperbolic settings (self-affine carpets) it is known that Gibbs-type measures may not necessarily exist or be unique \cite{kv} (which contrasts with the existence and uniqueness of conformal measures on attractors of finite conformal hyperbolic IFS). This would suggest that a more subtle criteria might determine the existence and uniqueness of weak Gibbs-type measures than just the decay rates of $\Phi(i^n)$ for parabolic indices $i \in \I$, and it would be interesting to investigate this further.

\section*{Acknowledgements}

The authors thank Mike Todd for some helpful discussions, especially surrounding Lemma \ref{bdp}.

\end{document}